\documentclass[a4paper,reqno,10pt]{amsart}

\usepackage{amsmath}
\usepackage{amsthm}
\usepackage{amssymb}
\usepackage{amscd} %diagrams

\usepackage{bbm} %bold numbers
\usepackage{mathrsfs} %nice calligraphy

\usepackage{lmodern} 
\usepackage[T1]{fontenc} %allows for hyphenation of words with accents
\usepackage{newtxtext,newtxmath}

\usepackage{manfnt} % The Bourbaki Z = \dbend

\usepackage{graphicx} %uncomment for inclusion of graphic images

\usepackage{verbatim} %DO NOT comment this line (comment environment)

\usepackage[Lenny]{fncychap} % other style = Rejne

\usepackage{sqrcaps} % square capital fonts

\usepackage{comment}

\usepackage[svgnames]{xcolor}
\usepackage{pdfcolmk}

\usepackage{tikz}
\usepackage{pifont}

\swapnumbers %theorems numbering appears in front of 'Theorem'

\newtheoremstyle{exercise} %for books or class notes
  {3pt} %space above
  {3pt} %space below
  {\small\rmfamily} %body font
  {
} %indent amount(empty=no indent,\parindent=para indent)
  {\rmfamily\scshape} %thm head font
  {.} %punctuation after thm head
  {.5em} %space after thm head: " " = normal interword space;
  {} %thm head spec (can be left empty, meaning `normal')

\newtheoremstyle{newplain}
  {5pt}
  {5pt}
  {\itshape}
  {}
  {\rmfamily\scshape}
  {. ---}
  {.5em}
  {}

\newtheoremstyle{newremark}
  {5pt}
  {5pt}
  {\rmfamily}
  {}
  {\rmfamily\scshape}
  {. ---}
  {.5em}
  {}

\theoremstyle{newplain}
\newtheorem*{Theorem*}{Theorem} %no numbering for Theorem*

\theoremstyle{newplain}
\newtheorem{Theorem}{Theorem}

\newtheorem{Corollary}[Theorem]{Corollary}
\newtheorem{Proposition}[Theorem]{Proposition}
\newtheorem{Conjecture}[Theorem]{Conjecture}
\newtheorem{Definition}[Theorem]{Definition}

\theoremstyle{newremark}
\newtheorem{Empty}[Theorem]{}
\newtheorem{Remark}[Theorem]{Remark}

\newtheorem{Claim}[Theorem]{Claim}

\theoremstyle{exercise}

\numberwithin{Theorem}{section}
\numberwithin{Exercise}{section}

\newcommand{\N}{\mathbb{N}} %natural numbers

\newcommand{\R}{\mathbb{R}} %real numbers
\newcommand{\Rm}{\R^m}
\newcommand{\Rn}{\R^n}
\newcommand{\ind}{\mathbbm{1}} %indicatrix function

\newcommand{\calA}{\mathscr{A}}
\newcommand{\calB}{\mathscr{B}}
\newcommand{\calC}{\mathscr{C}}

\newcommand{\calE}{\mathscr{E}}
\newcommand{\calF}{\mathscr{F}}

\newcommand{\calH}{\mathscr{H}}

\newcommand{\calL}{\mathscr{L}}

\newcommand{\calR}{\mathscr{R}}
\newcommand{\calS}{\mathscr{S}}

\newcommand{\calV}{\mathscr{V}}

\newcommand{\calY}{\mathscr{Y}}
\newcommand{\calZ}{\mathscr{Z}}

\newcommand{\frH}{\frak H}

\newcommand{\balpha}{\boldsymbol{\alpha}}

\newcommand{\bdelta}{\boldsymbol{\delta}}
\newcommand{\boldeta}{\boldsymbol{\eta}}

\newcommand{\bxi}{\boldsymbol{\xi}}

\newcommand{\bB}{\mathbf{B}}
\newcommand{\bC}{\mathbf{C}}
\newcommand{\bE}{\mathbf{E}}

\newcommand{\bG}{\mathbf{G}}
\newcommand{\bN}{\mathbf{N}}

\newcommand{\bU}{\mathbf{U}}
\newcommand{\bV}{\mathbf{V}}
\newcommand{\bW}{\mathbf{W}}

\newcommand{\bY}{\mathbf{Y}}

\newcommand{\bc}{\mathbf{c}}

\newcommand{\bv}{\mathbf{v}}
\newcommand{\bw}{\mathbf{w}}

\DeclareMathOperator{\rmBdry}{\mathrm{Bdry}} %boundary
\DeclareMathOperator{\rmClos}{\mathrm{Clos}} %closure
\DeclareMathOperator{\rmdiam}{\mathrm{diam}} %diameter

\DeclareMathOperator{\rmdist}{\mathrm{dist}} %distance
\DeclareMathOperator{\rmHom}{\mathrm{Hom}} %homomorphisms
\DeclareMathOperator{\rmid}{\mathrm{id}} %the identity map
\DeclareMathOperator{\rmim}{\mathrm{im}} %image
\DeclareMathOperator{\rmInt}{\mathrm{Int}} %interior
\DeclareMathOperator{\rmLip}{\mathrm{Lip}} %Lipschitz constant

\DeclareMathOperator{\rmrank}{\mathrm{rank}}
\DeclareMathOperator{\rmspan}{\mathrm{span}} %span
\DeclareMathOperator{\rmspt}{\mathrm{spt}} %support
\DeclareMathOperator{\rmtrace}{\mathrm{trace}} %trace

\newcommand{\rmI}{\mathrm{I}}
\newcommand{\rmII}{\mathrm{II}}

\newcommand{\lseg}{\boldsymbol{[}\!\boldsymbol{[}}
\newcommand{\rseg}{\boldsymbol{]}\!\boldsymbol{]}}

\newcommand{\hel} {
\hskip2.5pt{\vrule height7pt width.5pt depth0pt}
\hskip-.2pt\vbox{\hrule height.5pt width7pt depth0pt}
\, }

\newcommand{\bin}[2]{
\begin{pmatrix} #1 \\
#2 
\end{pmatrix}}

\def\Xint#1{\mathchoice
{\XXint\displaystyle\textstyle{#1}}%
{\XXint\textstyle\scriptstyle{#1}}%
{\XXint\scriptstyle\scriptscriptstyle{#1}}%
{\XXint\scriptscriptstyle
\scriptscriptstyle{#1}}%
\!\int}
\def\XXint#1#2#3{{%
\setbox0=\hbox{$#1{#2#3}{\int}$}
\vcenter{\hbox{$#2#3$}}\kern-.5\wd0}}

\def\dashint{\Xint-}

\newcommand{\veps}{\varepsilon}

\newcommand{\vphi}{\varphi}
\newcommand{\wh}{\widehat}
\newcommand{\la}{\langle}
\newcommand{\ra}{\rangle}

\renewcommand{\leq}{\leqslant}
\renewcommand{\geq}{\geqslant}
\renewcommand{\subset}{\subseteq}

\renewcommand{\bV}{\mathbb{V}}
\renewcommand{\bG}{\mathbb{G}}
\renewcommand{\bY}{Y}
\newlength{\drop}

\usepackage{trajan}

\begin{document}

%=================
% TITLE AND AUTHOR
%=================

%\titleAT %see above / page de garde d'un livre

\title[On a conjecture of A. Zygmund]{Density estimate from below\\
in relation to a conjecture of A. Zygmund\\
on Lipschitz differentiation}

\def\curraddrname{{\itshape On leave of absence from}}

\author[Th. De Pauw]{Thierry De Pauw}
\address{School of Mathematical Sciences\\
Shanghai Key Laboratory of PMMP\\ 
East China Normal University\\
500 Dongchuang Road\\
Shanghai 200062\\
P.R. of China\\
and NYU-ECNU Institute of Mathematical Sciences at NYU Shanghai\\
3663 Zhongshan Road North\\
Shanghai 200062\\
China}
\curraddr{Universit\'e Paris Diderot\\ 
Sorbonne Universit\'e\\
CNRS\\ 
Institut de Math\'ematiques de Jussieu -- Paris Rive Gauche, IMJ-PRG\\
F-75013, Paris\\
France}
\email{thdepauw@math.ecnu.edu.cn,thierry.de-pauw@imj-prg.fr}


%; Secondary 52A40,49J45}

\thanks{The author was partially supported by the Science and Technology Commission of Shanghai (No. 18dz2271000).}

%\date{December 9th, 2018}

%=========
% ABSTRACT
%=========

\begin{abstract}
Letting $A \subset \Rn$ be Borel measurable and $\bW_0 : A \to \bG(n,m)$ Lipschitzian, we establish that 
\begin{equation*}
\limsup_{r \to 0^+} \frac{\calH^m \left[ A \cap \bB(x,r) \cap (x+ \bW_0(x))\right]}{\balpha(m)r^m} \geq \frac{1}{2^n},
\end{equation*}
for $\calL^n$-almost every $x \in A$. In particular, it follows that
$A$ is $\calL^n$-negligible if and only if $\calH^m(A \cap (x+\bW_0(x))=0$, for $\calL^n$-almost every $x \in A$.  
\end{abstract}

\maketitle

%===========
% DEDICATION
%===========

\begin{comment}
\begin{flushright}
\textsf{\textit{Dedicated to: Joseph Plateau, Arno, Hooverphonic, Hercule
Poirot,\\
Kim Clijsters, Pierre Rapsat, and non Jef t'es pas tout seul}}
\end{flushright}
\end{comment}

%==================
% TABLE OF CONTENTS
%==================

\tableofcontents
%\newpage

%==============================
% THE MEAT --- OR SO ONE THINKS
%==============================

\section{Foreword}

Let $\calS_x$ be the set of squares centered at $x \in \R^2$.
The Lebesgue density theorem states that if $f : \R^2 \to \R$ is Lebesgue summable, then 
\begin{equation}
\label{eq.intro}
f(x) = \lim_{\substack{S \in \calS_x\\\rmdiam S \to 0^+}} \dashint_S f d\calL^2,\tag{$*$}
\end{equation}
for $\calL^2$-almost every $x \in \R^2$.
This consequence of the Vitali covering theorem fails if $\calS_x$ is replaced with $\calR_x$, the set of rectangles centered at $x$ of arbitrary direction and eccentricity. 
It fails even for some indicator function $f = \ind_A$.
Indeed, Nikod\'ym \cite{NIK.27} defined a set $A \subset [0,1] \times [0,1]$ of full measure with the following property: For every $x \in A$, there exists a line $L(x)$ such that $A \cap L(x) = \{x\}$.
In fact, Nikod\'ym's example shows that the Lebesgue density theorem fails if we replace $\calS_x$ with $\calR_x^L$, the set of rectangles of arbitrary eccentricity, centered at $x$, and one side of which is parallel to $L(x)$.
Furthermore \cite[Chap. IV Theorem 3.5]{GUZMAN.1981}, replacing $A$ by a nonnegligible subset of $A$, one may choose $L$ to be continuous with respect to $x$.
Yet, if $L$ is constant, then the Lebesgue density theorem with respect to $\calR_x^L$ holds, for $p$-summable functions $f$, $1 < p \leq \infty$, by virtue of a theorem of Zygmund, though the corresponding version of the Vitali covering theorem fails, according to an example of H. Bohr \cite[Chap. IV Theorem 1.1]{GUZMAN.1981}.
This raises the question: What regularity condition of $x \mapsto L(x)$ guarantees that the Lebesgue density theorem holds with respect to $\calR_x^L$ for, say, functions that are square summable?
The following version is a conjecture reportedly \cite{LAC.LI.10} attributed to Zygmund.
{\it 
Assume that $x \mapsto L(x)$ is a Lipschitzian field of lines, with $x \in L(x)$, and $f : \R^2 \to \R$ is square summable.
Is it true that
\begin{equation*}
f(x) = \lim_{r \to 0^+} \dashint_{L(x) \cap \bB(x,r)} f d\calH^1,
\end{equation*}
for $\calL^2$-almost every $x \in \R^2$?
Here, $\calH^1$ is the 1-dimensional Hausdorff measure.
}
\par 
E.M. Stein raised the singular integral variant of this conjecture.
Both have received much attention from the harmonic analysis community.
To the author's knowledge, most results recorded so far in the literature, via the maximal function approach, assume some extra regularity property of $L$ -- namely, that $L$ be $C^1$ together with a hypothesis on the variation of the derivative -- see, for instance,  \cite{BOU.89} and \cite{LAC.LI.10}.
\par 
In this paper we offer a novel approach -- a kind of change of variable based on an appropriate fibration and the coarea formula.
This allows for treating the case when $L$ is Lipschitzian, without further restriction.
We obtain the following lower density bound for indicator functions.
{\it 
Let $x \mapsto L(x)$ be a Lipschitzian field of lines, with $x \in L(x)$, and let $A \subset \R^2$ be Lebesgue measurable.
Then,
\begin{equation*}
\limsup_{r \to 0^+} \frac{\calH^1(A \cap B(x,r)\cap L(x) )}{2r} \geq \frac{1}{4},
\end{equation*}
for $\calL^2$-almost every $x \in A$.
}
Incidentally, the following corollary -- a nonparallel version of Fubini theorem -- seems to be new as well.
{\it 
The set $A$ is Lebesgue negligible if and only if $\calH^1(A \cap L(x))=0$, for $\calL^2$-almost every $x \in \R^2$.
}
\par 
Our results hold in any dimension and codimension. 

\nocite{BRU.71}

\section{Sketch of proof}

Let $A$ be a subset of Euclidean space $\Rn$, $n \geq 2$, and let $\calL^n$ be the Lebesgue outer measure. 
We start by considering the following weak question: Can one tell whether $A$ is Lebesgue negligible from the knowledge only of its trace on each member of some given collection of ``lower dimensional'' subsets $\Gamma_i \subset \Rn$, $i \in I$. 
One expects that if $A \cap \Gamma_i$ is ``negligible in the dimension of $\Gamma_i$'', for each $i \in I$, then $\calL^n(A)=0$. 
Of course, a necessary condition is that the sets $\Gamma_i$ cover almost all of $A$, i.e., $\calL^n(A \setminus \cup_{i \in I} \Gamma_i)=0$. 
Consider, for instance, $n=2$, $I = \R$, and $\Gamma_t = \{t\} \times \R$, for $t \in \R$, the collection of all vertical lines in the plane. 
It is not true in general that if $A \subset \R^2$ and $A \cap \Gamma_t$ is a singleton, for each $t \in \R$, then $\calL^2(A)=0$. 
There exist, indeed, functions $f : \R \to \R$ whose graph $A$ satisfies $\calL^2(A) > 0$ -- see, e.g., \cite[Chapter 2 Theorem 4]{KHARAZISHVILI} for an example due to Sierpi\'nski.
In order to rule out such examples, we will henceforth assume that $A \subset \Rn$ is Borel measurable. 
In that case, the theorem of Fubini, together with the invariance of the Lebesgue measure under orthogonal transformations, imply the following. 
Given an integer $1 \leq m \leq n-1$, if $(\Gamma_i)_{i \in I}$ is the collection of all $m$-dimensional affine subspaces of $\Rn$ of some fixed direction, and if $\calH^m(A \cap \Gamma_i)=0$ for all $i \in I$, then $\calL^n(A)=0$. 
Here, $\calH^m$ denotes the $m$-dimensional Hausdorff measure. 
A special feature of this collection $(\Gamma_i)_{i \in I}$ is that it partitions $\Rn$, its members being the level sets $f^{-1}\{y\}$, $y \in \R^{n-m}$, of a ``nice map'' $f : \Rn \to \R^{n-m}$, indeed, an orthogonal projection.
This is an occurrence of the following more general situation when $f$ and its leaves $f^{-1}\{y\}$ are allowed to be nonlinear. 
The coarea formula due to Federer \cite{FED.59} asserts that if $f : \Rn \to \R^{n-m}$ is Lipschitzian and $A \subset \Rn$ is Borel measurable, then
\begin{equation*}
\int_A Jf(x) d\calL^n(x) = \int_{\R^{n-m}} \calH^m\left(A \cap f^{-1}\{y \}\right) d \calL^{n-m}(y) .
\end{equation*} 
Thus, if the Jacobian coarea factor $Jf$ is positive, $\calL^n$-almost everywhere in $A$, then the collection $\left( f^{-1}\{y\} \right)_{y \in \R^{n-m}}$ is suitable for detecting whether or not $A$ is Lebesgue null. 
At $\calL^n$-almost all $x \in \Rn$, the map $f$ is differentiable, by virtue of Rademacher's theorem, and
\begin{equation*}
Jf(x) = \sqrt{ \left| \det \left( Df(x) \circ Df(x)^*\right)\right|} = \left\| \wedge_{n-m} Df(x) \right\|,
\end{equation*}
see \cite[Chapter 3 \S 4]{EVANS.GARIEPY} and \cite[3.2.1 and 3.2.11]{GMT}.
\par
In this paper, we focus on the case when $\Gamma_i$, $i \in I$, are affine subspaces of $\Rn$, but not necessarily members of a partition of the ambient space. 
Specifically, we assume that with each $x \in \Rn$ is associated an $m$-dimensional affine subspace $\bW(x)$ of $\Rn$ containing $x$. 
Given a Borel set $A \subset \Rn$, the question whether
\begin{equation}
\label{eq.1}
\textit{If } \calH^m(A \cap \bW(x))=0 \text{, for all } x \in A \textit{, then } \calL^n(A)=0 \,,
\end{equation}  
has a negative answer, in view of Nikod\'ym's set $A \subset \R^2$ evoked in the previous section.
Indeed, corresponding to this set $A$, there exists a field of lines $x \mapsto \bW(x)$ such that $A \cap \bW(x) = \{x\}$, for all $x \in A$.
In this context, a selection theorem due to von Neumann implies that (possibly considering a smaller, non Lebesgue null, Borel subset of $A$) the correspondence $x \mapsto \bW(x)$ can be chosen to be Borel measurable (see \ref{nik.set}) and, in turn, it can be chosen to be continuous, according to Lusin's theorem. 
Nonetheless, when $\bW$ is Lipschitzian, the situation improves, as illustrated in our theorem below; $\bG(n,m)$ is the Grassmannian manifold consisting of $m$-dimensional vector subspaces of $\Rn$.

\begin{Theorem*}
Assume $\bW_0 : \Rn \to \bG(n,m)$ is Lipschitzian and $A \subset \Rn$ is Borel measurable. The following are equivalent.
\begin{enumerate}
\item[(A)] $\calL^n(A)=0$.
\item[(B)] For $\calL^n$ almost every $x \in A$, $\calH^m\left(A \cap (x+\bW_0(x))\right)=0$.
\item[(C)] For $\calL^n$ almost every $x \in \Rn$, $\calH^m\left(A \cap (x+\bW_0(x))\right)=0$.
\end{enumerate}
\end{Theorem*}

As should be apparent from the discussion above, one difficulty stands with the fact that the affine $m$-planes $\bW(x) = x + \bW_0(x)$ may not be disjointed.
Nevertheless, they locally are, in the following sense.
Given $x_0 \in A$ there exist a neighborhood $U$ of $x_0$ and Lipschitzian maps $\bw_1,\ldots,\bw_m : U \to \Rn$ so that $\bw_1(x),\ldots,\bw_m(x)$ is an orthonomal frame spanning $\bW_0(x)$, for $x \in U$, and the map $\Phi : (V_{x_0} \cap U) \times \Rm \to \Rn : (\xi,t) \mapsto \xi + \sum_{i=1}^m t_i \bw_i(\xi)$ is a lipeomorphism of a neighborhood of $x_0$ onto its image -- here, $V_{x_0} = x_0 + V$ and $V \in \bG(n,n-m)$ is close to $\bW_0(x_0)^\perp$.
This, and applications of Fubini's theorem, yield $(B) \Rightarrow (A)$ in the theorem above (see \ref{remark}).
\par 
However, we aim at obtaining a quantative version of this result, that the change of variable just described does not seem to provide.
A natural route is to reduce the problem to applying the coarea formula, by spreading out the $\bW(x)$'s in a disjointed way, in a higher dimensional space -- i.e., adding a variable $u \in \bW(x)$ to the given $x \in \Rn$ and considering $\bW(x)$ as a fiber above the base space $\Rn$.
We thus define
\begin{equation*}
\Sigma = \Rn \times \Rn \cap \{ (x,u) : x \in E \text{ and } u \in \bW(x) \} ,
\end{equation*}
where $E \subset \Rn$ is Borel measurable. 
This set is $(n+m)$-rectifiable, owing to the Lipschitz continuity of $\bW$. 
It is convenient to assume that $\calL^n(E) < \infty$, so that
\begin{equation}
\label{eq.01}
\phi_E(B) = \int_E \calH^m \left( B \cap \bW(x) \right) d\calL^n(x) ,
\end{equation}
$B \subset \Rn$, is a locally finite Borel measure \ref{def.phi}. 
Now, $\Sigma$ was precisely set up so that, for each $x \in E$,
\begin{equation*}
\calH^m \left( \Sigma \cap \pi_2^{-1}(B) \cap \pi_1^{-1}\{x\} \right) = \calH^m \left( B \cap \bW(x)\right) ,
\end{equation*}
where $\pi_1$ and $\pi_2$ denote the projections of $\Rn \times \Rn$ to the $x$ and $u$ variable, respectively.
Abbreviating $\Sigma_B =  \Sigma \cap \pi_2^{-1}(B)$, the coarea formula yields
\begin{equation}
\label{eq.02}
\phi_E(B) = \int_{\Sigma_B} J_\Sigma \pi_1 d\calH^{n+m} .
\end{equation}
A simple calculation \ref{factor.pi.2} shows that $J_\Sigma \pi_2 > 0$, $\calH^{n+m}$-almost everywhere on $\Sigma_B$. Since also 
\begin{equation}
\label{eq.03}
\int_{\Sigma_B} J_\Sigma \pi_2 d \calH^{n+m} = \int_B \calH^m \left( \Sigma_B \cap \pi_2^{-1}\{u\}\right)d\calL^n(u),
\end{equation}
the implication $(A) \Rightarrow (C)$ above should now be clear: Letting $B=A$ and $E=\bB(0,R)$, one infers from hypothesis (A) and \eqref{eq.03} that $\calH^{n+m}(\Sigma_B) = 0$, whence, $\phi_{\bB(0,R)}(A)=0$, by \eqref{eq.02}, and, in turn, conclusion (C) ensues from \eqref{eq.01}.
\par 
In order to establish that $(B) \Rightarrow (A)$ by using this new change of variable, we need to observe \ref{factor.pi.1} that $J_\Sigma \pi_1 > 0$, $\calH^{n+m}$-almost everywhere, and, ideally, to show that $\calH^m \left( \Sigma_B \cap \pi_2^{-1}\{u\}\right) > 0$, for almost every $u \in B$. 
This last part involves some difficulty. 
To understand this, we let $m=n-1$, in order to keep the notations short. 
Now $u \in \bW(x)$ iff $u-x \in \bW_0(x)$ iff $\la \bv_0(x) , x-u \ra =0$, where $\bv_0(x) \in \bW_0(x)^\perp$ is, say, a unit vector. 
Abbreviating $g_u(x) = \la \bv_0(x) , x-u \ra$, we infer that
\begin{equation*}
\calH^m \left( \Sigma_B \cap \pi_2^{-1}\{u\}\right) = \calH^m \left( E \cap g_u^{-1} \{0\}\right).
\end{equation*}
The problem remains that two of the nonlinear $m$ sets $E \cap g_u^{-1} \{0\}$ and $E \cap g_{u'}^{-1} \{0\}$ may intersect, thereby preventing another application of the coarea formula when looking out for their lower bound. 
Yet, we already know that
\begin{equation*}
\phi_E(B) = \int_B \calZ_E \bW d\calL^n,
\end{equation*}
where $\calZ_E \bW$ is a Radon-Nikod\'ym derivative (see \ref{def.Z}) and also that $(\calZ_E \bW)(u)$ is comparable to $\calH^m \left( E \cap g_u^{-1} \{0\}\right)$ (see \ref{Z.1}). 
Adding an extra variable $y$ to the fibered space $\Sigma$, \ref{fibration.2}, we improve on this by showing that
\begin{equation*}
(\calZ_E \bW)(u) \geq \boldeta(n,m) \liminf_j \dashint_{-j^{-1}}^{j^{-1}} \calH^m \left( E \cap g_u^{-1} \{y\}\right) d\calL^1(y) = \boldeta(n,m) (\calY_E \bW)(u),
\end{equation*}
where the last equality defines $\calY_E \bW$, and $\boldeta(n,m) > 0$ is a lower bound for the coarea Jacobian factor of the restriction to the fibered space of the projection $(x,u,y) \mapsto (x,y)$ (see \ref{def.Y} and \ref{lb.2}). 
We are reduced to showing that $\calY_E \bW > 0$, almost everywhere. 
The reason why this holds is the following. 
Fix a Borel set $Z \subset \Rn$, $x_0 \in \Rn$ and $r > 0$. 
Let $\bC_\bW(x_0,r)$ be the cylindrical box consisting of those $x \in \Rn$ such that $\left|P_{\bW_0(x_0)}(x-x_0) \right| \leq r$ and $\left|P_{\bW_0(x_0)^\perp}(x-x_0) \right| \leq r$. 
We want to find a lower bound for
\begin{equation*}
\int_{Z \cap \bC_\bW(x_0,r)} \left(\calY_{Z \cap \bC_\bW(x_0,r)}\right)(u) d\calL^n(u).
\end{equation*}
To this end, we fix $z \in \bW_0(x_0) \cap \bB(0,r)$ and we let $V_z = \Rn \cap \{ x_0 + z + s\bv_0(x_0) : -r \leq s \leq r \}$ denote the corresponding vertical line segment. 
According to Fubini's theorem we are reduced to estimating
\begin{equation*}
\int_{V_z} \left(\calY_{Z \cap \bC_\bW(x_0,r)}\right)(u) d\calH^1(u) .
\end{equation*}
By virtue of Vitali's covering theorem, we can find a disjointed family of line segments $I_1,I_2,\ldots$, covering almost all $V_z$, such that the above integral nearly equals
\begin{multline*}
\sum_k \calH^1(I_k) \dashint_{I_k} \calH^m \left( Z \cap \bC_\bW(x_0,r) \cap g_{u_k}^{-1} \{y\}\right) d\calL^1(y) \\
\cong \sum_k \int_{ Z \cap \bC_\bW(x_0,r) \cap g_{u_k}^{-1} (I_k) } \left| \nabla g_{u_k}(x) \right| d\calL^n(x) \cong \calL^n \left( Z \cap \bC_\bW(x_0,r)\right),
\end{multline*}
where the first near equality follows from the coarea formula, and the second one because $\left| \nabla g_{u_k} \right| \cong 1$ at small scales (see \ref{jac.g}) and the ``nonlinear horizontal stripes'' $g_{u_k}^{-1} (I_k)$ are nearly pairwise disjoint. 
Verification of these claims takes up sections 6 and 7. 
Now, we reach a contradiction if $Z = \Rn \cap \{ \calY_E \bW = 0\}$ is assumed to have $\calL^n(Z) > 0$, and $x_0$ is a point of density of $Z$.
\par 
Along this path, we have, in fact, achieved a quantitative version of these observations.
Indeed, assuming that $\limsup_{r \to 0^+} \frac{\calH^m (A \cap \bB(x,r) \cap \bW(x))}{\balpha(m)r^m} < \frac{\boldeta(n,m)}{2^m}$, for $\calL^n$ positively many $x \in A$, we start by choosing $E \subset A$ with the same property, whose diameter is small enough for the various estimates above to hold, and we refer to Egoroff's theorem to select $\veps > 0$ so that the set
\begin{equation*}
Z = E \cap \left\{ \limsup_{r \to 0^+} \frac{\calH^m (A \cap \bB(x,r) \cap \bW(x))}{\balpha(m)r^m} < (1-\veps) \frac{\boldeta(n,m)}{(2+\veps)^m} \right\}
\end{equation*}
has positive $\calL^n$ measure. 
Thus, for $0 < \hat{\veps} \ll \veps$, there exists $x_0 \in Z$ and there exists $r > 0$ as small as we please so that
\begin{multline*}
\left( 1 - \hat{\veps} \right) \calL^m \left( \bC_\bW(x_0,r) \right) \leq \calL^n  \left(Z \cap  \bC_\bW(x_0,r) \right) \lesssim \int_{Z \cap \bC_\bW(x_0,r)} \frac{\calY_{Z \cap \bC_\bW(x_0,r)} \bW(u)}{\balpha(m)r^m} d\calL^n(u) \\
\lesssim \frac{1}{\boldeta(n,m)} \int_{Z \cap \bC_\bW(x_0,r)} \frac{\calH^m(Z \cap \bC_\bW(x_0,r) \cap \bW(x))}{\balpha(m)r^m	} d\calL^n(x)\\
 \leq \frac{1}{\boldeta(n,m)} \int_{Z \cap \bC_\bW(x_0,r)} \frac{\calH^m(Z \cap \bB(x,2r+o(r)) \cap \bW(x))}{\balpha(m)r^m	} d\calL^n(x) \\
 \lesssim \frac{1}{\boldeta(n,m)} \int_{Z \cap \bC_\bW(x_0,r)}  (1-\veps) \frac{\boldeta(n,m)}{(2+\veps)^m} d\calL^n(x) \leq (1-C\veps) \calL^n \left( \bC_\bW(x_0,r) \right) ,
\end{multline*}
where the symbol $\lesssim$ means that we loose a multiplicative factor $\lambda > 1$ as close as we wish to 1, according to the scale $r$. 
The above contradiction proves our main result \ref{main.density}, namely,
\begin{Theorem*}
Assume $\bW_0 : \Rn \to \bG(n,m)$ is Lipschitzian and $A \subset \Rn$ is Borel measurable. Abbreviating $\bW(x) = x + \bW_0(x)$, it follows that
\begin{equation*}
\limsup_{r \to 0^+} \frac{\calH^m \big( A \cap \bB(x,r) \cap \bW(x) \big)}{\balpha(m)r^m} \geq \frac{1}{2^n},
\end{equation*}
for $\calL^n$-almost every $x \in A$.
\end{Theorem*}
\par 
I extend my warm thanks to Jean-Christophe L\'eger for carefully reading an early version of the manuscript. 

%{\color{Blue} I am also grateful to the referee for their very useful remarks and suggestions.}

\section{Preliminaries}

\begin{Empty}
In this paper, $1 \leq m \leq n-1$ are integers. 
The ambient space is $\Rn$. 
The canonical inner product of $x,x' \in \Rn$ is denoted $\la x,x' \ra$ and the corresponding Euclidean norm of $x$ is $|x|$. 
If $S \subset \Rn$, we let $\calB(S)$ be the $\sigma$-algebra of Borel measurable subsets of $S$.
\end{Empty}

\begin{Empty}[Hausdorff measure]
\label{h.m}
We let $\calL^n$ be the Lebesgue outer measure in $\Rn$ and $\balpha(n) = \calL^n(\bB(0,1))$. 
For $S \subset \Rn$, we abbreviate $\zeta^m(S) = \balpha(m) 2^{-m}(\rmdiam S)^m$. 
Given $0 < \delta \leq \infty$, we call $\delta$-cover of $A \subset \Rn$ a finite or countably infinite family $(S_j)_{j \in J}$ of subsets of $\Rn$ such that $A \subset \cup_{j \in J} S_j$ and $\rmdiam S_j \leq \delta$, for every $j \in J$. 
We define
\begin{equation*}
\calH^m_\delta(A) = \inf \left\{ \sum_{j \in J} \zeta^m(S_j) : (S_j)_{j \in J} \text{ is a $\delta$-cover of $A$ }\right\}
\end{equation*}
and $\calH^m(A) = \lim_{\delta \to 0^+} \calH^m_\delta(A) = \sup_{\delta > 0} \calH^m_\delta(A)$.
Thus, $\calH^m$ is the $m$-dimensional Hausdorff outer measure in $\Rn$. 
\par 
{\it 
\begin{enumerate}
\item[(1)] If $(K_k)_k$ is a sequence of nonempty compact subsets of $\Rn$ converging in Hausdorff distance to $K$, then $\calH^m_\infty(K) \geq \limsup_k \calH^m_\infty(K_k)$.
\end{enumerate}
}
\par 
Given $\veps > 0$, choose an $\infty$-cover $(S_j)_{j \in \N}$ of $K$ such that $\calH^m_\infty(K) + \veps \geq \sum_j \zeta^m(S_j)$. 
Since $\lim_{r \to 0^+} \zeta^m(\bU(S_j,r))=\zeta^m(S_j)$, for each $j \in \N$, we can choose an open set $U_j$ containing $S_j$ such that $\zeta^m(U_j) \leq  \veps 2^{-j} + \zeta^m(S_j)$. 
Since $U = \cup_j U_j$ is open, there exists a positive integer $k_0$ such that $K_k \subset U$, whenever $k \geq k_0$. 
Thus, in that case $(U_j)_j$ is an $\infty$-cover of $K_k$ and, therefore, $\calH^m_\infty(K_k) \leq \sum_j \zeta^m(U_j) \leq 2\veps + \calH^m_\infty(K)$. 
Take the limit superior of the left hand side, as $k \to \infty$, and then let $\veps \to 0$.
\par 
{\it 
\begin{enumerate}
\item[(2)] For all $A \subset \Rm$, one has $\calL^m(A) = \calH^m(A) = \calH^m_\infty(A)$.
\end{enumerate}
}
\par
It suffices to note that $\calH^m(A) \geq \calH^m_\infty(A) \geq \calL^m(A) \geq \calH^m(A)$. 
The first inequality is trivial; the second one follows from the isodiametric inequality \cite[2.10.33]{GMT}; the last one is a consequence of the Vitali covering theorem \cite[Chapter 2 \S 2 Theorem 2]{EVANS.GARIEPY}.
\par 
{\it 
\begin{enumerate}
\item[(3)] If $W \subset \Rn$ is an $m$-dimensional affine subspace and $A \subset W$, then $\calH^m(A) = \calH^m_\infty(A)$.
\end{enumerate}
}
\par
Let $\frH^m$ denote the $m$-dimensional Hausdorff outer measure in the metric space $W$. 
In other words,
\begin{equation*}
\frH^m_\delta(A) = \inf \left\{ \sum_{j \in J} \zeta^m(S_j) : (S_j)_{j \in J} \text{ is a $\delta$- cover of $A$ and $S_j \subset W$, for all $j \in J$}\right\} ,
\end{equation*}
and $\frH^m(A) = \sup_{\delta > 0} \frH^m_\delta(A)$.
It is elementary to observe that $\frH^m(A) = \calH^m(A)$ and that $\frH^m_\infty(A) = \calH^m_\infty(A)$. 
Now, if $f : W \to \Rm$ is an isometry, then $\calH^m(A) = \frH^m(A) =  \calH^m(f(A)) = \calH^m_\infty(f(A)) = \frH^m_\infty(A) = \calH^m_\infty(A)$, where the third equality follows from claim (2) above.
\end{Empty}

\begin{Empty}[Coarea formula]
\label{coarea}
We recall two versions of the coarea formula. 
First, if $A \subset \Rn$ is $\calL^n$-measurable and $f : A \to \R^{n-m}$ is Lipschitzian, then $\R^{n-m} \to [0,\infty] : y \mapsto \calH^m \left( A \cap f^{-1}\{y\} \right)$ is $\calL^{n-m}$- measurable and 
\begin{equation*}
\int_A Jf(x) d\calL^n(x) = \int_{\R^{n-m}} \calH^m\left(A \cap f^{-1}\{y \}\right) d \calL^{n-m}(y) .
\end{equation*} 
Here, the coarea Jacobian factor is well-defined $\calL^n$-almost everywhere, according to Rademacher's theorem, and equals
\begin{equation*}
Jf(x) = \sqrt{ \left| \det \left( Df(x) \circ Df(x)^*\right)\right|} = \left\| \wedge_{n-m} Df(x) \right\| ;
\end{equation*}
see, for instance, \cite[Chapter 3 \S 4]{EVANS.GARIEPY}. 
\par 
Secondly, if $A \subset \R^p$ is $\calH^n$-measurable and countably $(\calH^n,n)$-rectifiable and if $f : A \to \R^{n-m}$ is Lipschitzian, then $\R^{n-m} \to [0,\infty] : y \mapsto \calH^m \left( A \cap f^{-1}\{y\} \right)$ is $\calL^{n-m}$-measurable and
\begin{equation*}
\int_A J_Af(x) d\calH^n(x) = \int_{\R^{n-m}} \calH^m\left(A \cap f^{-1}\{y \}\right) d \calL^{n-m}(y) .
\end{equation*} 
In order to state a formula for the coarea Jacobian factor $J_Af(x)$ of $f$ relative to $A$, we consider a point $x \in A$, where $A$ admits an approximate $n$-dimensional tangent space $T_xA$ and $f$ is differentiable at $x$ along $A$. 
Letting $L : T_x A \to \R^{n-m}$ denote the derivative of $f$ at $x$, we have
\begin{equation*}
J_Af(x) = \sqrt{ \left| \det \left( L \circ L^*\right)\right|} = \left\| \wedge_{n-m} L \right\| ;
\end{equation*}
see, for instance, \cite[3.2.22]{GMT}.
\par 
In both cases, it is useful to recall the following. 
If $L : V \to V'$ is a linear map between two inner product spaces $V$ and $V'$, then
\begin{equation}
\label{eq.11}
\|\wedge_k L \| = \sup \left\{ \la \wedge_k L , \xi \ra : \xi \in \wedge_k V \text{ and } |\xi| = 1 \right\} .
\end{equation}
On the one hand, $\|\wedge_k L\| \leq \|L\|^k$ \cite[1.7.6]{GMT} and $\|L\| \leq \rmLip f$, with $L$ as above. 
On the other hand, if $v_1,\ldots,v_k$ are linearly independent vectors of $V$, then 
\begin{equation}
\label{eq.12}
\|\wedge_k L \| \geq \frac{|L(v_1)\wedge \ldots \wedge L(v_k)|}{|v_1 \wedge \ldots \wedge v_k|} .
\end{equation}
\par 
Finally, we observe that both coarea {formul\ae} hold true when $f$ is merely locally Lipschitzian, according to the monotone convergence theorem.
\end{Empty}

\begin{Empty}[Grassmannian]
\label{grassmann}
We let $\bG(n,m)$ be the set whose members are the $m$-dimensional vector subspaces of $\Rn$. 
With $W \in \bG(n,m)$, we associate $P_W : \Rn \to \Rn$, the orthogonal projection on $W$. 
We give $\bG(n,m)$ the structure of a compact metric space by letting $d(W_1,W_2) = \|P_{W_1}-P_{W_2}\|$, where $\|\cdot\|$ is the operator norm. 
If $W \in \bG(n,m)$, then $W^\perp \in \bG(n,n-m)$ satisfies $P_W + P_{W^\perp}=\rmid_{\Rn}$, therefore, $\bG(n,m) \to \bG(n,n-m) : W \mapsto W^\perp$ is an isometry. 
The bijective correspondence $\vphi :\bG(n,m) \to \rmHom(\Rn,\Rn) : W \mapsto P_W$ identifies $\bG(n,m)$ with the submanifold $M_{n,m} = \rmHom(\Rn,\Rn) \cap \{ L : L \circ L = L\,,\, L^*=L \text{, and } \rmtrace L = m \}$. 
There exists an open neighborhood $V$ of $M_{n,m}$ in $\rmHom(\Rn,\Rn)$ and a Lipschitzian retraction $\rho : V \to M_{n,m}$, according, for instance, to \cite[3.1.20]{GMT}. 
Therefore, if $S \subset \Rn$ and $\bW_0 : S \to \bG(n,m)$ is Lipschitzian, then there exist an open neighborhood $U$ of $S$ in $\Rn$ and a Lipschitzian extension $\wh{\bW}_0 : U \to \bG(n,m)$ of $\bW_0$. 
Indeed, $\vphi \circ \bW_0$ admits a Lipschitzian extension $\bY : \Rn \to \rmHom(\Rn,\Rn)$ -- see, e.g., \cite[2.10.43]{GMT} -- and it suffices to let $U = \bY^{-1}(V)$ and $\wh{\bW}_0 = \rho \circ \left( \bY |_U \right)$.
\end{Empty}

\begin{Empty}[Orthonormal frames]
\label{stiefel}
We let $\bV(n,m)$ be the set of orthonormal $m$-frames in $\Rn$, i.e., $\bV(n,m) = (\Rn)^{m} \cap\{ (w_1,\ldots,w_{m}) : \text{ the family } w_1,\ldots,w_{m} \text{ is orthonormal}\}$. 
We will consider it as a metric space, with its structure inherited from $(\Rn)^{m}$. 
\end{Empty}

\begin{Empty}
\label{grass.param.stief}
{\it 
Let $\calV \subset \bG(n,m)$ be a nonempty closed set such that $\rmdiam \calV < 1$. There exists a Lipschitzian map $\Xi : \calV  \to \bV(n,m)$ such that $W = \rmspan \{ \Xi_1(W),\ldots,\Xi_m(W)\}$, for every $W \in \calV$. 
}
\end{Empty}

\begin{proof}
Pick arbitrarily $W_0 \in \calV$. 
If $W \in \calV$, then the map $W_0 \to W : w \mapsto P_W(w)$ is bijective: If $w \in W_0 \setminus \{0\}$, then $\left|P_W(w)-w\right| = \left|P_W(w)-P_{W_0}(w)\right| < |w|$, thus, $P_W(w) \neq 0$.
Letting $w_1,\ldots,w_m$ be an arbitrary basis of $W_0$, it follows that, for each $W \in \calV$, the vectors $w_i(W) = P_W(w_i)$, $i=1,\ldots,m$, constitute a basis of $W$. 
Furthermore, the maps $w_i : \calV \to \Rn$ are Lipschitzian: $\left| w_i(W)-w_i(W')\right| = \left|P_W(w_i) - P_{W'}(w_i)\right| \leq d(W,W') |w_i|$. 
We apply the Gram-Schmidt process:
\begin{equation*}
\overline{w}_1(W) = w_1(W) \quad\text{ and }\quad \overline{w}_i(W) = w_i(W) - \sum_{j=1}^{i-1} \la w_i(W) , \overline{w}_j(W) \ra \overline{w}_j(W) \,,\,\, i=2,\ldots,m,
\end{equation*}
so that $\overline{w}_1(W),\ldots,\overline{w}_m(W)$ is readily an orthogonal basis of $W$ depending upon $W$ in a Lipschitzian way. 
Since each $\left|\overline{w}_i\right|$ is bounded away from zero on $\calV$, the formula $\Xi_i(W) = \left|\overline{w}_i(W)\right|^{-1}\overline{w}_i(W)$, $i=1,\ldots,m$, defines $\Xi$ with the required property.
\end{proof}

\begin{Empty}
\label{grass.param.stief.borel}
{\it 
There exists a Borel measurable map $\Xi : \bG(n,m)  \to \bV(n,m)$ with the property that $W = \rmspan \{ \Xi_1(W),\ldots,\Xi_m(W)\}$, for every $W \in \bG(n,m)$. 
}
\end{Empty}

\begin{proof}
Since $\bG(n,m)$ is compact, it can partitioned into finitely many Borel measurable sets $\calV_1,\ldots,\calV_J$, each having diameter bounded by $1/2$. 
Define $\Xi$, piecewise, to coincide on $\calV_j$ with $\Xi_j$ associated to $\rmClos \calV_j$ in \ref{grass.param.stief}, $j=1,\ldots,J$.
\end{proof}

\begin{Empty}
\label{orth.frame}
{\it 
Assume $S \subset \Rn$, $x_0 \in S$, and $\bW_0 : S \to \bG(n,m)$ is Lipschitzian. There exist an open neighborhood $U$ of $x_0$ in $\Rn$ and Lipschitzian maps $\bw_1,\ldots,\bw_m,\bv_1,\ldots,\bv_{n-m} : U \to \Rn$ with the following properties.
\begin{enumerate}
\item[(1)] For every $x \in U$, the family $\bw_1(x),\ldots,\bw_m(x),\bv_1(x),\ldots,\bv_{n-m}(x)$ is an orthonormal basis of $\Rn$.
\item[(2)] For every $x \in S \cap U$, one has
\begin{equation*}
\bW_0(x) = \rmspan \{ \bw_1(x),\ldots,\bw_m(x) \}
\end{equation*}
and
\begin{equation*}
\bW_0(x)^\perp = \rmspan \{ \bv_1(x),\ldots,\bv_{n-m}(x) \} .
\end{equation*}
\end{enumerate}
}
\end{Empty}

\begin{proof}
We let $\wh{\bW}_0 : \wh{U} \to \bG(n,m)$ be a Lipschitzian extension of $\bW_0$, where $\wh{U}$ is an open neighborhood of $S$ in $\Rn$ (recall \ref{grassmann}). 
Abbreviate $W_0 := \bW_0(x_0)$. 
Define $\calV = \bG(n,m) \cap \left\{ W : d(W,W_0) < 1/4 \right\}$ and $U = \wh{U} \cap \wh{\bW}_0^{-1}(\calV)$. 
Apply \ref{grass.param.stief} to $\rmClos \calV$ and denote by $\Xi$ the resulting Lipschitzian map $\calV \to (\Rn)^m$. 
Next, define $\calV^\perp = \bG(n,n-m) \cap \left\{ W^\perp : W \in \calV \right\}$, apply \ref{grass.param.stief} to $\rmClos \calV^\perp$, and denote by $\Xi^\perp$ the resulting Lipschitzian map $\calV^\perp \to (\Rn)^{n-m}$. 
Letting $\bw_i(x) = \left( \Xi_i \circ \wh{\bW}_0 \right)(x)$, $i=1,\ldots,m$, and $\bv_i(x) = \left( \Xi^\perp_i \circ \wh{\bW}_0 \right)(x)$, $i=1,\ldots,n-m$, completes the proof.
\end{proof}

\begin{Empty}
\label{orth.frame.borel}
{\it 
Assume $\bW_0 : \Rn \to \bG(n,m)$ is Borel measurable. There exist Borel measurable maps $\bw_1,\ldots,\bw_m,\bv_1,\ldots,\bv_{n-m} : \Rn \to \Rn$ with the following properties.
\begin{enumerate}
\item[(1)] For every $x \in \Rn$, the family $\bw_1(x),\ldots,\bw_m(x),\bv_1(x),\ldots,\bv_{n-m}(x)$ is an orthonormal basis of $\Rn$.
\item[(2)] For every $x \in\Rn$, one has
\begin{equation*}
\bW_0(x) = \rmspan \{ \bw_1(x),\ldots,\bw_m(x) \}
\end{equation*}
and
\begin{equation*}
\bW_0(x)^\perp = \rmspan \{ \bv_1(x),\ldots,\bv_{n-m}(x) \} .
\end{equation*}
\end{enumerate}
}
\end{Empty}

\begin{proof}
Choose $\Xi : \bG(n,m) \to \bV(n,m)$ and $\Xi^\perp : \bG(n,n-m) \to \bV(n,n-m)$ as in \ref{grass.param.stief.borel}. 
Letting $\left(\bw_1(x),\ldots,\bw_m(x)\right) = \left( \Xi \circ \bW_0 \right)(x)$ and $\left(\bv_1(x),\ldots,\bv_{n-m}(x)\right) = \left( \Xi^\perp \circ \bW_0^\perp \right)(x)$, for $x \in \Rn$, completes the proof.
\end{proof}

\begin{Empty}[Definition of $\bW(x)$]
\label{W}
The typical situation that arises in the remaining part of this paper is that we are given a set $S \subset \Rn$, a Lipschitzian map $\bW_0 : S \to \bG(n,m)$, and $x_0 \in S$. 
We will represent $\bW_0(x)$ and $\bW_0^\perp(x)$ in a neighborhood $U$ of $x_0$ as in \ref{orth.frame}. 
We will then further reduce the size of $U$ several times, in order that various conditions be met.
With no exception, we will denote as $\bW(x) = x + \bW_0(x)$ the affine subspace containing $x$, of direction $\bW_0(x)$, whenever $\bW_0(x)$ is defined.
\end{Empty}

\begin{Empty}[Definition of $g_{\bv_1,\ldots,\bv_{n-m},u}$ and lower bound of its coarea factor]
Given an open set $U \subset \Rn$, a Lipschitzian map $\bv : U \to \Rn$, and $u \in \Rn$, we define $g_{\bv,u} : U \to \R$ by the formula
\begin{equation*}
g_{\bv,u}(x) = \la \bv(x) , x-u  \ra .
\end{equation*} 
Clearly, $g_{\bv,u}$ is locally Lipschitzian. 
If $\bv$ is differentiable at $x \in U$, then so is $g_{\bv,u}$ and, for every $h \in \Rn$, one has
\begin{equation}
\label{eq.2}
Dg_{\bv,u}(x)(h) = \la \nabla g_{\bv,u}(x),h \ra = \la D\bv(x)(h) , x-u \ra + \la \bv(x) , h \ra .
\end{equation}
\par 
Next, we assume that we are given Lipschitzian maps $\bv_1,\ldots,\bv_{n-m} : U \to \Rn$. 
We define $g_{\bv_1,\ldots,\bv_{n-m},u} : U \to \R^{n-m}$ by the formula
\begin{equation*}
g_{\bv_1,\ldots,\bv_{n-m},u}(x) = \left( g_{\bv_1,u}(x) , \ldots , g_{\bv_{n-m},u}(x) \right) .
\end{equation*}
It is locally Lipschitzian as well. 
The relevance of $g_{\bv_1,\ldots,\bv_{n-m},u}$ stems from the following observation, assuming that $\bv_1,\ldots,\bv_{n-m}$ are associated with $\bW_0$ and $\bW$ as in \ref{orth.frame} and \ref{W}:
\begin{equation}
\label{eq.7}
\begin{split}
u \in \bW(x) & \iff u-x \in \bW_0(x) \\
& \iff \la \bv_i(x) , u-x \ra = 0 \text{, for all } i=1,\ldots,n-m \\
& \iff g_{\bv_1,\ldots,\bv_{n-m},u}(x)=0 \\
& \iff x \in g_{\bv_1,\ldots,\bv_{n-m},u}^{-1}\{0\} .
\end{split}
\end{equation}
In fact, $|g_{\bv_1,\ldots,\bv_{n-m},u}(x)| = \left| P_{\bW_0(x)^\perp}(x-u)\right|$.
\par 
Abbreviate $g = g_{\bv_1,\ldots,\bv_{n-m},u}$. 
If each $\bv_i$ is differentiable at $x \in U$, and $h \in \Rn$, then
\begin{equation*}
Dg(x)(h) = \sum_{i=1}^{n-m} Dg_{\bv_i,u}(x)(h)e_i ,
\end{equation*}
where, here and elsewhere, $e_1,\ldots,e_{n-m}$ denotes the canonical basis of $\R^{n-m}$.
Thus, if $\bv_1(x),\ldots,\bv_{n-m}(x)$ constitute an orthonormal family in $\Rn$, then
\begin{equation*}
Dg_{\bv_i,u}(x)(\bv_j(x)) = \delta_{i,j} + \veps_{i,j}(x,u),
\end{equation*}
where
\begin{equation}
\label{eq.3}
|\veps_{i,j}(x,u)| = \left| \la D\bv_i(x)(\bv_j(x)) , x-u \ra\right| \leq \left( \rmLip \bv_i \right) |x-u| ,
\end{equation}
according to \eqref{eq.2}, and, in turn,
\begin{equation*}
Dg(x)(\bv_j(x)) = \sum_{i=1}^{n-m} \left(\delta_{i,j} + \veps_{i,j}(x,u)\right) e_i .
\end{equation*}
This allows for a lower bound of the coarea factor of $g$ at $x$ as follows. 
\begin{equation*}
\begin{split}
\left\| \wedge_{n-m}Dg(x) \right\| & \geq \left| Dg(x)(\bv_1(x)) \wedge \ldots \wedge Dg(x)(\bv_{n-m}(x)) \right| \\
& = \left| \left( \sum_{i=1}^{n-m} \left(\delta_{i,1} + \veps_{i,1}(x,u)\right) e_i \right) \wedge \ldots \wedge  \left( \sum_{i=1}^{n-m} \left(\delta_{i,n-m} + \veps_{i,n-m}(x,u)\right) e_i \right) \right| \\
& = \left| \det \left( \delta_{i,j} + \veps_{i,j}(x,u) \right)_{i,j=1,\ldots,n-m} \right| .
\end{split}
\end{equation*}
In view of \eqref{eq.3}, we obtain the next lemma.
\end{Empty}

\begin{Empty}
\label{jac.g}
{\it 
Given $\Lambda > 0$ and $0 < \veps < 1$, there exists $\bdelta_{\theTheorem}(n,\Lambda,\veps) > 0$ with the following property. Assume that
\begin{enumerate}
\item[(1)] $U \subset \Rn$ is open and $u \in \Rn$;
\item[(2)] $\bv_1,\ldots,\bv_{n-m} : U \to \Rn$ are Lipschitzian;
\item[(3)] $\bv_1(x),\ldots,\bv_{n-m}(x)$ is an orthonormal family, for every $x \in U$.
\end{enumerate}
If
\begin{enumerate}
\item[(4)] $\rmLip \bv_i \leq \Lambda$, for each $i=1,\ldots,n-m$;
\item[(5)] $\rmdiam \left( U \cup \{u\} \right) \leq \bdelta_{\theTheorem}(n,\Lambda,\veps)$;
\end{enumerate}
then
\begin{equation*}
Jg_{\bv_1,\ldots,\bv_{n-m},u}(x) \geq 1 - \veps,
\end{equation*}
at $\calL^n$-almost every $x \in U$.
}
\end{Empty}

\begin{Empty}[Definition of $\pi_u$ and its relation with $g_{\bv_1,\ldots,\bv_{n-m},u}$]
\label{pi.u}
With $u \in \Rn$, we associate
\begin{equation*}
\pi_u : \bV(n,n-m) \times \Rn \to \R^{n-m} : (\xi_1,\ldots,\xi_{n-m},x) \mapsto\left( \la \xi_1,x-u\ra, \ldots, \la \xi_{n-m},x-u\ra \right).
\end{equation*}
When $(\xi_1,\ldots,\xi_{n-m}) \in \bV(n,n-m)$ is fixed, we also abbreviate as $\pi_{\xi_1,\ldots,\xi_{n-m},u}$ the map $\Rn \to \R^{n-m}$ defined by $\pi_{\xi_1,\ldots,\xi_{n-m},u}(x) = \pi_u(\xi_1,\ldots,\xi_{n-m},x)$. 
It is then rather useful to observe that, in the context described in \ref{orth.frame} and \ref{W}, the following holds:
\begin{equation}
\label{eq.4}
\pi_{\bv_1(x),\ldots,\bv_{n-m}(x),u}^{-1}\left\{ g_{\bv_1,\ldots,\bv_{n-m},u}(x) \right\} = \bW(x) .
\end{equation}
Indeed,
\begin{equation*}
\begin{split}
h \in \bW(x) & \iff h-x \in \bW_0(x) \\
& \iff \la \bv_i(x) , h-x \ra = 0 \text{, for all }  i=1,\ldots,n-m \\
& \iff \la \bv_i(x) , h-u \ra = \la \bv_i(x) , x-u \ra \text{, for all } i=1,\ldots,n-m \\
& \iff \pi_{\bv_1(x),\ldots,\bv_{n-m}(x),u}(h) = g_{\bv_1,\ldots,\bv_{n-m},u}(x) .
\end{split}
\end{equation*}
In the sequel, we will sometimes abbreviate $\xi = (\xi_1,\ldots,\xi_{n-m}) \in \bV(n,n-m)$.
It also helps to notice that, for given $\xi \in \bV(n,n-m)$ and $y \in \R^{n-m}$, the set $\pi_{\xi,u}^{-1}\{y\}$ is an $m$-dimensional affine subspace of $\Rn$.
\end{Empty}

\begin{Empty}
\label{measurability}
{\it 
If $B \in \calB(\Rn)$ and $u \in \Rn$, then
\begin{equation*}
h_{B} :  \bV(n,n-m) \times \R^{n-m} \to [0,\infty] : (\xi,y) \mapsto \calH^m\left(B  \cap \pi_{\xi,u}^{-1}\{y\} \right) 
\end{equation*}
is Borel measurable.
}
\end{Empty}

\begin{proof}
We start by showing that, when $B$ is compact, $h_{B}$ is upper semicontinuous. 
Thus, if $((\xi_k,y_k))_k$ is a sequence in $\bV(n,n-m) \times \R^{n-m}$ that converges to $(\xi,y)$, we ought to show that
\begin{equation}
\label{eq.5}
\calH^m_\infty(K) \geq \limsup_k \calH^m_\infty(K_k),
\end{equation}
where $K = B  \cap \pi_{\xi,u}^{-1}\{y\}$ and $K_k = B  \cap \pi_{\xi_k,u}^{-1}\{y_k\}$. 
This is indeed equivalent to the same inequality with $\calH^m_\infty$ replaced by $\calH^m$, according to \ref{h.m}(3) and the last sentence of \ref{pi.u}. 
Considering, if necessary, a subsequence of $(K_k)_k$ we may assume that none of the compact sets $K_k$ is empty and that the limit superior in \eqref{eq.5} is a limit. 
Since the set of nonempty compact subsets of the compact set $B$, equipped with the Hausdorff metric, is compact, the sequence $(K_k)_k$ admits a subsequence (denoted the same way) converging to a compact set $L \subset B$. 
Given $z \in L$, there are $z_k \in K_k$ converging to $z$. 
Thus, $\pi_u(\xi,z) = \lim_k \pi_u(\xi_k,z_k) = \lim_k y_k = y$. 
In other words, $z \in K$. 
Therefore, $\calH^m_\infty(K) \geq \calH^m_\infty(L)$, so that \eqref{eq.5} follows from \ref{h.m}(1).
\par 
Next, we fix $r > 0$ and we abbreviate $\calA = \calB(\Rn) \cap \{ B : h_{B \cap \bB(0,r)} \text{ is Borel measurable}\}$. 
Thus, we have just shown that $\calA$ contains the collection $\calF(\Rn)$ of all closed subsets of $\Rn$. 
Observe that if $(B_j)_j$ is an nondecreasing sequence in $\calA$ and $B = \cup_j B_j$, then $h_{B \cap \bB(0,r)} = \lim_j h_{B_j \cap \bB(0,r)}$ pointwise, thus, $B \in \calA$. 
In particular, $\Rn \in \calA$. 
Furthermore, if $B,B' \in \calA$ and $B' \subset B$, then $h_{B \setminus B'} = h_{B} - h_{B'}$, because all measures involved are finite; indeed, $h_{B \cap \bB(0,r)}(\xi,y) \leq \balpha(m)r^m$, for all $(\xi,y)$. 
Accordingly, $B \setminus B' \in \calA$. 
This means that $\calA$ is a Dynkin class. 
Since $\calF(\Rn)$ is a $\pi$-system, $\calA$ contains the $\sigma$-algebra generated by $\calF(\Rn)$, i.e., $\calB(\Rn)$ \cite[Theorem 1.6.2]{COHN}. 
Finally, if $B \in \calB(\Rn)$, then $h_B = \lim_j h_{B \cap \bB(0,j)}$ pointwise, whence $h_B$ is Borel measurable. 
\end{proof}

\begin{Empty}
\label{second.measurability}
{\it 
Assume $B \in \calB(\Rn)$ and $\bW_0 : \Rn \to \bG(n,m)$ is Borel measurable. The following function is Borel measurable.
\begin{equation*}
\Rn \to [0,\infty] : x \mapsto \calH^m\left( B \cap \bW(x) \right).
\end{equation*}
}
\end{Empty}

\begin{proof}
Let $h_{\bW,B}$ denote this function. 
Let $\bv_1,\ldots,\bv_{n-m} : \Rn \to \Rn$ be Borel measurable maps associated with $\bW_0$ in \ref{orth.frame.borel}. 
Fix $u \in \Rn$ arbitrarily. 
Define
\begin{equation*}
\Upsilon : \Rn \to \bV(n,n-m) \times \R^{n-m} : x \mapsto \left(\bv_1(x),\ldots,\bv_{n-m}(x),g_{\bv_1(x),\ldots,\bv_{n-m}(x),u}(x) \right)
\end{equation*}
and notice that
\begin{equation*}
h_{\bW,B} = h_{B} \circ \Upsilon
\end{equation*}
(where $h_{B}$ is the function associated with $B$ and $u$ in \ref{measurability}), according to \eqref{eq.4}. 
One notes that $\Upsilon$ is Borel measurable; whence, the conclusion ensues from \ref{measurability}.
\end{proof}

\begin{Empty}[Definition of $\phi_{E,\bW}$]
\label{def.phi}
Let $\bW_0 : \Rn \to \bG(n,m)$ be Borel measurable and $E \in \calB(\Rn)$ be such that $\calL^n(E) < \infty$. 
For each $B \in \calB(\Rn)$, we define
\begin{equation*}
\phi_{E,\bW}(B) = \int_E \calH^m \left( B \cap \bW(x) \right) d\calL^n(x) .
\end{equation*}
This is well-defined, according to \ref{second.measurability}. 
It is easy to check that $\phi_{E,\bW}$ is a locally finite -- hence, $\sigma$-finite -- Borel measure on $\Rn$; indeed, $\phi_{E,\bW}(B) \leq \balpha(m) (\rmdiam B)^m \calL^n(E)$. 
\par 
To close this section, we discuss the relevance of $\phi_{E,\bW}$ to the problem of existence of ``nearly Nikod\'ym sets''.
\end{Empty}

\begin{Empty}[Definition of Nearly Nikod\'ym set]
Let $E \in \calB(\Rn)$. We say that $B \in \calB(E)$ is {\bf nearly $m$-Nikod\'ym in $E$} if
\begin{enumerate}
\item[(1)] $\calL^n(B) > 0$;
\item[(2)] For $\calL^n$-almost each $x \in E$, there is $W \in \bG(n,m)$ such that $\calH^m\left( B \cap (x+W) \right) = 0$.
\end{enumerate}
\par
In case $n=2$, $m=1$, and $E=[0,1] \times [0,1]$, the existence of such a $B$ (with $\calL^2(B)=1$) was established by Nikod\'ym \cite{NIK.27}, see also \cite[Chapter 8]{GUZMAN.1981}. 
For arbitrary $n \geq 2$ and $m=n-1$, the existence of such a $B$ was established by Falconer \cite{FAL.86}. 
In fact, in both cases, these authors established the stronger condition that, for every $x \in B$, $\calH^m(B \cap (x+W))=0$ can be replaced by $B \cap (x+W) = \{x\}$. 
Thus, in case $1 \leq m< n-1$, letting $B$ be a set exhibited by Falconer, if $x \in B$ and $W \subset \bG(n,n-1)$ is such that $B \cap (x+W) = \{x\}$, then picking arbitrarily $V \in \bG(n,m)$ such that $V \subset W$, we see that $B \cap (x+V)=\{x\}$. 
Whence, $B$ is also nearly $m$-Nikod\'ym in $B$. 
\par 
Assuming that $\bW_0 : E \to \bG(n,m)$ is Borel measurable, we say that $B \in \calB(E)$ is {\bf nearly $m$-Nikod\'ym in $E$ relative to $\bW$} if
\begin{enumerate}
\item[(1)] $\calL^n(B) > 0$;
\item[(2)] For $\calL^n$-almost each $x \in E$, one has $\calH^m\left( B \cap \bW(x) \right) = 0$.
\end{enumerate}
\end{Empty}

\begin{Empty}
{\it 
Let $E \in \calB(\Rn)$ have finite $\calL^n$ measure and $\bW_0 : \Rn \to \bG(n,m)$ be Borel measurable. The following are equivalent.
\begin{enumerate}
\item[(1)] $\calL^n |_{\calB(E)}$ is absolutely continuous with respect to $\phi_{E,\bW}|_{\calB(E)}$.
\item[(2)] There does not exist a nearly $m$-Nikod\'ym set in $E$ relative to $\bW$.
\end{enumerate}
}
\end{Empty}

\begin{proof}
A set $B \in \calB(E)$ such that $\phi_{E,\bW}(B)=0$ and $\calL^n(B) > 0$ is, by definition, a nearly $m$-Nikod\'ym set relative to $\bW$. 
Condition (1) is equivalent to their nonexistence.
\end{proof}

\begin{Empty}
\label{nik.set}
{\it 
Assume that $E \in \calB(\Rn)$ and that $B \in \calB(E)$ is nearly $m$-Nikod\'ym. The following hold.
\begin{enumerate}
\item[(1)] There exists $\bW_0 : \Rn \to \bG(n,m)$ Borel measurable such that $B$ is nearly $m$- Nikod\'ym in $E$ relative to $\bW$.
\item[(2)] There exists $C \subset B$ compact and $\overline{\bW}_0 : \Rn \to \bG(n,m)$ continuous such that $C$ is nearly $m$-Nikod\'ym in $C$ relative to $\overline{\bW}$.
\end{enumerate}
}
\end{Empty}

\begin{proof}
Define a Borel measurable map $\bxi : \bG(n,m) \to \bV(n,n-m)$ by $\bxi(W) = \Xi \left( W^\perp \right)$, where $\Xi : \bG(n,n-m) \to \bV(n,n-m)$ is as in \ref{grass.param.stief.borel}. 
Choose arbitrarily $u \in \Rn$ and define a Borel measurable map
\begin{multline*}
\Upsilon :  E \times \bG(n,m) \to  \bV(n,n-m) \times \R^{n-m} \\ (x,W) \mapsto \left(\bxi(W) , \la \bxi_1(W),x-u  \ra , \ldots, \la \bxi_{n-m}(W),x-u \ra \right) .
\end{multline*} 
Similarly to \eqref{eq.4}, observe that
\begin{equation*}
x+W = \pi_{\bxi(W),u}^{-1} \left\{ \left( \la \bxi_1(W),x-u  \ra , \ldots, \la \bxi_{n-m}(W),x-u \ra \right) \right\},
\end{equation*}
for every $(x,W) \in E \times \bG(n,m)$. 
We infer from \ref{measurability} that
\begin{equation*}
 h_{B} \circ \Upsilon :E \times \bG(n,m) \to [0,\infty]:(x,W) \mapsto \calH^m \left( B  \cap (x+W) \right)
\end{equation*}
is Borel measurable. 
Thus, the set
\begin{equation*}
\calE = E \times \bG(n,m) \cap \left\{ (x,W) : \calH^m \left( B \cap (x+W) \right) = 0 \right\}
\end{equation*}
is Borel measurable as well. 
The set $N = E \cap \{ x : \calE_x = \emptyset \}$ is coanalytic and $\calL^n(N)=0$, by assumption.
By virtue of von Neumann's selection theorem \cite[5.5.3]{SRIVASTAVA}, there exists a universally measurable map $\tilde{\bW}_0 : E \setminus N \to \bG(n,m)$ such that $\tilde{\bW}_0(x) \in \calE_x$, for every $x \in E \setminus N$, i.e., $\calH^m \left( B \cap \left( x + \tilde{\bW}_0(x) \right) \right) = 0$. 
We extend $\tilde{\bW}_0$ to be an arbitrary constant on $N \cup (\Rn \setminus E)$. 
This makes $\tilde{\bW}_0$ an $\calL^n$-measurable map defined on $\Rn$. 
Therefore, it is equal $\calL^n$-almost everywhere to a Borel measurable map $\bW_0 : \Rn \to \bG(n,m)$. 
This proves (1).
\par 
In order to prove (2), we recall \ref{grassmann}, specifically, the retraction $\rho : V \to M_{n,m}$ and the homeomorphic identification $\vphi : \bG(n,m) \to M_{n,m}$. 
Owing to the compactness of $M_{n,m}$, there are finitely many open balls $U_j$, $j=1,\ldots,J$, whose closures are contained in $V$ and whose union contains $M_{n,m}$. 
Since $\calL^n(B) > 0$, there exists $j \in \{1,\ldots,J\}$ such that $\calL^n \left( B \cap E_j \right) > 0$, where $E_j = \left( \vphi \circ \bW_0 \right)^{-1}(U_j)$. 
It follows from Lusin's theorem \cite[2.5.3]{GMT} that there exists a compact set $C \subset B \cap E_j$ such that $\calL^n(C) > 0$ and the restriction $\bW_0|_C$ is continuous. 
The map $\vphi \circ \bW_0|_C$ takes its values in the closed ball $\rmClos U_j$, therefore, it admits a continuous extension $\bY_0 : \Rn \to \rmClos U_j \subset V$. 
Letting $\overline{\bW}_0 = \vphi^{-1} \circ \rho \circ \bY_0$ completes the proof.
\end{proof}

\section{Common setting}

\begin{Empty}[Setting for the next three sections]
\label{31}
In the next three sections, we shall assume the following.
\begin{enumerate}
\item $E \subset \Rn$ is Borel measurable and $\calL^n(E) < \infty$.
\item $U \subset \Rn$ is open and $E \subset U$.
\item $B \subset \Rn$ is Borel measurable.
\item $\bW_0 : U \to \bG(n,m)$ is Lipschitzian.
\item $\bW(x) = x + \bW_0(x)$, for each $x \in U$.
\item $\Lambda > 0$.
\item $\bw_1,\ldots,\bw_m : U \to \Rn$ and $\rmLip \bw_i \leq \Lambda$, $i=1,\ldots,m$.
\item $\bv_1,\ldots,\bv_{n-m} : U \to \Rn$ and $\rmLip \bv_i \leq \Lambda$, $i=1,\ldots,n-m$.
\item $\bW_0(x) = \rmspan \{ \bw_1(x),\ldots,\bw_m(x)\}$, for every $x \in U$.
\item $\bW_0(x)^\perp = \rmspan \{ \bv_1(x),\ldots,\bv_{n-m}(x)\}$, for every $x \in U$.
\item $\bw_1(x),\ldots,\bw_m(x),\bv_1(x),\ldots,\bv_{n-m}(x)$ constitute an orthonormal basis of $\Rn$, for every $x \in U$.
\end{enumerate}
\end{Empty}

\section{Two fibrations}

\begin{Empty}[A fibered space associated with $E,B,\bw_1,\ldots,\bw_m$]
\label{fibration.1}
We define 
\begin{equation*}
F : E \times \Rm \to \Rn \times \Rn : (x,t_1,\ldots,t_m) \mapsto \left( x , x + \sum_{i=1}^m t_i \bw_i(x) \right)
\end{equation*}
as well as 
\begin{equation*}
\Sigma = F( E \times \Rm) = \Rn \times \Rn \cap \left\{ (x,u) : x \in E \text{ and } u \in \bW(x) \right\}.
\end{equation*}
We will oftentimes abbreviate $(x,t)=(x,t_1,\ldots,t_m)$.
It is obvious that $F$ is locally Lipschitzian and, therefore, $\Sigma$ is countably $(n+m)$-rectifiable and $\calH^{n+m}$-measurable. 
We also consider the two canonical projections
\begin{equation*}
\pi_1 : \Rn \times \Rn \to \Rn : (x,u) \mapsto x \quad\text{ and }\quad \pi_2 : \Rn \times \Rn \to \Rn : (x,u) \mapsto u,
\end{equation*}
as well as the set
\begin{equation*}
\Sigma_B = \Sigma \cap \pi_2^{-1}(B) = \Rn \times \Rn \cap \left\{ (x,u) : x \in E \text{ and } u \in B \cap \bW(x) \right\},
\end{equation*}
which also is, clearly, countably $(n+m)$-rectifiable and $\calH^{n+m}$-measurable. 
With the prospect of applying the coarea formula to $\Sigma_B$ and $\pi_1$, and to $\Sigma_B$ and $\pi_2$, respectively, we observe that, for each fixed $x \in E$,
\begin{equation*}
\Sigma_B \cap \pi_1^{-1}\{x\} = \{x\} \times \big( \Rn \cap \{ u : u \in B \cap \bW(x) \}\big) ,
\end{equation*}
so that
\begin{equation}
\label{eq.6}
\calH^m \left(\Sigma_B \cap \pi_1^{-1}\{x\} \right) = \calH^m \left( B \cap \bW(x)\right) ,
\end{equation}
and that, for each fixed $u \in B$,
\begin{equation*}
\begin{split}
\Sigma_B \cap \pi_2^{-1}\{u\} & =  \big( \Rn \cap \{ x : x \in E \text{ and } u \in \bW(x) \} \big)\times \{u\}\\
& = \left( \Rn\cap \left\{ x : x \in E \cap g_{\bv_1,\ldots,\bv_{n-m},u}^{-1}\{0\} \right\} \right) \times \{u\},
\end{split}
\end{equation*}
according to \eqref{eq.7}, so that
\begin{equation}
\label{eq.8}
\calH^m \left( \Sigma_B \cap \pi_2^{-1}\{u\}\right) = \calH^m \left( E \cap g_{\bv_1,\ldots,\bv_{n-m},u}^{-1}\{0\}\right) ,
\end{equation}
whenever $u \in B$. 
It now follows from the coarea formula that
\begin{equation}
\label{eq.9}
\int_{\Sigma_B} J_\Sigma \pi_1 d\calH^{n+m} = \int_E \calH^m \left( B \cap \bW(x) \right) d\calL^n(x) = \phi_{E,\bW}(B)
\end{equation}
and
\begin{equation}
\label{eq.10}
\int_{\Sigma_B} J_\Sigma \pi_2 d\calH^{n+m} = \int_B \calH^m \left( E \cap g_{\bv_1,\ldots,\bv_{n-m},u}^{-1}\{0\}\right) d\calL^n(u) .
\end{equation}
For these formul\ae{} to be useful, we need to establish bounds for the coarea Jacobian factors $J_\Sigma \pi_1$ and $J_\Sigma \pi_2$.
In order to do so, we notice that if $\Sigma \ni (x,u) = F(x,t)$, $F$ is differentiable at $(x,t)$, i.e., each $\bw_i$ is differentiable at $x$, $i=1,\ldots,m$, then the approximate tangent space $T_{(x,u)}\Sigma$ exists and is generated by the following $n+m$ vectors of $\Rn \times \Rn$:
\begin{equation*}
\begin{split}
\sum_{p=1}^n \la w_j , e_p \ra \frac{\partial F}{\partial x_p}(x,t) &  = \left( w_j, w_j + \sum_{i=1}^m t_i D \bw_i(x)(w_j) \right) \,, \,j=1,\ldots,n, \\
\frac{\partial F}{\partial t_k}(x,t) & = \left( 0 , \bw_k(x) \right) \,, \,k=1,\ldots,m ,
\end{split}
\end{equation*}
where $w_1,\ldots,w_n$ is an arbitrary basis of $\Rn$.
As usual, $e_1,\ldots,e_n$ denotes the canonical basis of $\Rn$.
\end{Empty}

\begin{Empty}[Coarea Jacobian factor of $\pi_1$] 
\label{factor.pi.1}
{\it 
For $\calH^{n+m}$-almost every $(x,u) \in \Sigma$, one has
\begin{equation*}
 \left( 1 + m^2 \Lambda^2 |x-u|^2 \right)^{-\frac{m}{2}}\left( 2 + 2m\Lambda |x-u| + m^2\Lambda^2 |x-u|^2 \right)^{-\frac{n-m}{2}} 
\leq J_\Sigma \pi_1(x,u) \leq 1 .
\end{equation*}
}
\end{Empty}

\begin{proof}
We recall \ref{coarea}.
That the right hand inequality be valid follows from $\rmLip \pi_1 = 1$. 
Regarding the left hand inequality, fix $(x,u) = F(x,t)$ such that $F$ is differentiable at $(x,t)$ and let $L : T_{(x,u)}\Sigma \to \R^n$ denote the restriction of $\pi_1$ to $T_{(x,u)}\Sigma$. 
Define $w_j = \bw_j(x)$, $j=1,\ldots,m$, and $w_j = \bv_{j-m}(x)$, $j=m+1,\ldots,n$.
Put 
\begin{equation*}
v_j = \sum_{p=1}^n \la w_j,e_p \ra\frac{\partial F}{\partial x_p}(x,t) - \frac{\partial F}{\partial t_j}(x,t) = \left( w_j , \sum_{i=1}^m t_i D\bw_i(x)(w_j) \right) ,
\end{equation*}
$j=1,\ldots,m$, and
\begin{equation*}
v_j = \sum_{p=1}^n \la w_j,e_p \ra\frac{\partial F}{\partial x_p}(x,t) = \left( w_j , w_j + \sum_{i=1}^m t_i D\bw_i(x)(w_j) \right) ,
\end{equation*}
$j=m+1,\ldots,n$.
Recall \eqref{eq.12} that
\begin{equation*}
J_\Sigma \pi_1(x,u) = \| \wedge_n L \| \geq \frac{1}{|v_1 \wedge \ldots \wedge v_n|},
\end{equation*}
since $L(v_j)=w_j$, $j=1\ldots,n$. 
Now, notice that, for $j=1,\ldots,m$,
\begin{equation*}
\left| v_j\right|^2 = |w_j|^2 + \left|  \sum_{i=1}^m t_i D \bw_i(x)(w_j)\right|^2 
\leq 1 + m^2\Lambda^2|t|^2 ,
\end{equation*}
whereas, for $j=m+1,\ldots,n$,
\begin{multline*}
\left| v_j \right|^2 = |w_j|^2 + \left| w_j + \sum_{i=1}^m t_i D \bw_i(x)(w_j)\right|^2 
\leq 2 + 2 \left|\sum_{i=1}^m t_i D \bw_i(x)(w_j) \right| + \left|\sum_{i=1}^m t_i D \bw_i(x)(w_j) \right|^2 \\
\leq 2 + 2 m \Lambda |t| + m^2\Lambda^2|t|^2 .
\end{multline*}
Since $u = x + \sum_{i=1}^m t_i \bw_i(x)$, one also has
\begin{equation*}
|u-x|^2 = \left| \sum_{i=1}^m t_i \bw_i(x) \right|^2 = |t|^2 .
\end{equation*}
Finally,
\begin{multline*}
|v_1 \wedge \ldots \wedge v_n | 
 \leq \left( 1 + m^2 \Lambda^2 |x-u|^2 \right)^\frac{m}{2}\left( 2 + 2 m \Lambda |x-u| + m^2\Lambda^2|x-u|^2 \right)^\frac{n-m}{2},
\end{multline*}
and the conclusion follows.
\end{proof}

\begin{Empty}
\label{34}
{\it 
Let $1 \leq q \leq n-1$ be an integer and let $v_1,\ldots,v_q$ be an orthonormal family in $\Rn$. There exists $\lambda \in \Lambda(n,q)$ such that 
\begin{equation*}
\left| \det \left( \la v_k , e_{\lambda(j)} \ra \right)_{j,k=1,\ldots,q}\right| \geq \bin{n}{q}^{-\frac{1}{2}} \,.
\end{equation*}
Here, $\Lambda(n,q)$ denotes the set of increasing maps $\{1,\ldots,q\} \to \{1,\ldots,n\}$.
}
\end{Empty}

\begin{proof}
We define a linear map $L : \R^q \to \Rn : (s_1,\ldots,s_q) \mapsto \sum_{k=1}^q s_k v_k$ and we observe that $L$ is an isometry. 
Therefore, its area Jacobian factor $JL=1$, by definition. 
Now, also
\begin{equation*}
(JL)^2 = \sum_{\lambda \in \Lambda(n,q)}\left| \det \left( \la v_k , e_{\lambda(j)} \ra \right)_{j,k=1,\ldots,q}\right|^2,
\end{equation*}
according to the Binet-Cauchy formula \cite[Chapter 3 \S 2 Theorem 4]{EVANS.GARIEPY}. 
The conclusion easily follows.
\end{proof}

\begin{Empty}[Coarea Jacobian factor of $\pi_2$] 
\label{factor.pi.2}
{\it 
The following hold.
\begin{enumerate}
\item 
For $\calH^{n+m}$-almost every $(x,u) \in \Sigma$, one has
\begin{equation*}
\frac{\bin{n}{n-m}^{-\frac{1}{2}} - m(n-m) \Lambda |x-u| \left( 1 + m \Lambda |x-u|\right)^{n-m-1}}{\left( 2 + 2 m \Lambda |x-u| + m^2\Lambda^2|x-u|^2 \right)^\frac{n-m}{2}} \leq J_\Sigma \pi_2(x,u) \leq 1 .
\end{equation*}
\item 
For $\calH^{n+m}$-almost every $(x,u) \in \Sigma$, one has $J_\Sigma \pi_2(x,u) > 0$.
\end{enumerate}
}
\end{Empty}

\begin{proof}
Clearly, $J_\Sigma \pi_2(x,u) \leq (\rmLip \pi_2)^n \leq 1$. 
Regarding the left hand inequality, fix $(x,u) = F(x,t)$ such that $F$ is differentiable at $(x,t)$ and, this time, let $L : T_{(x,u)}\Sigma \to \R^n$ denote the restriction of $\pi_2$ to $T_{(x,u)}\Sigma$. 
We will now define a family of $n$ vectors $v_1,\ldots,v_n$ belonging to $T_{(x,u)}\Sigma$. 
We choose $v_k = \frac{\partial F}{\partial t_k}(x,t) = (0,\bw_k(x))$, for $k=1,\ldots,m$. 
For choosing the $n-m$ remaining vectors, we proceed as follows. 
We select $\lambda \in \Lambda(n,n-m)$ as in \ref{34} applied with $q=n-m$ to $\bv_1(x),\ldots,\bv_{n-m}(x)$ and we let $v_{m+j} = \frac{\partial F}{\partial x_{\lambda(j)}}(x,t)$, $j=1,\ldots,n-m$. 
Recalling \eqref{eq.12}, we have
\begin{equation*}
J_\Sigma \pi_1(x,u) = \| \wedge_n L \| \geq \frac{|L(v_1) \wedge \ldots \wedge L(v_n)|}{|v_1 \wedge \ldots \wedge v_n|} .
\end{equation*}
As in the proof of \ref{factor.pi.1}, we find that
\begin{equation*}
|v_1 \wedge \ldots \wedge v_n| \leq \left( 2 + 2 m \Lambda |x-u| + m^2\Lambda^2|x-u|^2 \right)^\frac{n-m}{2}
\end{equation*}
and it remains only to find a lower bound for $|L(v_1) \wedge \ldots \wedge L(v_n)|$. 
The latter equals the absolute value of the determinant of the matrix of coefficients of $L(v_i)$, $i=1,\ldots,n$, with respect to any orthonormal basis of $\Rn$. 
We choose the basis $\bw_1(x),\ldots,\bw_m(x),\bv_1(x),\ldots,\bv_{n-m}(x)$. 
Thus,
\begin{equation}
\label{eq.13}
\begin{split}
|L(v_1) \wedge \ldots \wedge L(v_n)| & = \left| \det \left(
\begin{array}{c c c | c}
1 & \cdots & 0 & * \\
\vdots & \ddots & \vdots & \vdots \\
0 & \cdots & 1 & * \\
\hline
0 & \cdots & 0 & \left\la e_{\lambda(j)} + \sum_{i=1}^m t_i \frac{\partial \bw_i}{\partial x_{\lambda(j)}}(x) , \bv_k(x) \right\ra \\
\end{array}
\right)\right| \\
& = \left| \det \left( \left\la e_{\lambda(j)} + \sum_{i=1}^m t_i \frac{\partial \bw_i}{\partial x_{\lambda(j)}}(x) , \bv_k(x) \right\ra\right)_{j,k=1,\ldots,n-m}\right| .
\end{split}
\end{equation}
Abbreviate
\begin{equation*}
h_{\lambda(j)} = \sum_{i=1}^m t_i \frac{\partial \bw_i}{\partial x_{\lambda(j)}}(x)
\end{equation*}
and observe that $\left| h_{\lambda(j)} \right| \leq m \Lambda |t| = m \Lambda |x-u|$, $j=1,\ldots,n-m$ (recall the proof of \ref{factor.pi.1}). 
It remains only to remember that $\lambda$ has been selected in order that
\begin{equation*}
\left| \det \left( \left\la e_{\lambda(j)}, \bv_k(x)\right\ra\right)_{j,k=1,\ldots,n-m} \right| \geq \bin{n}{n-m}^{-\frac{1}{2}}
\end{equation*}
and to infer from the multilinearity of the determinant that
\begin{multline*}
\left| \det \left( \left\la e_{\lambda(j)} , \bv_k(x)\right\ra + \left\la h_{\lambda(j)} , \bv_k(x)\right\ra\right)_{j,k} - \det \left( \left\la e_{\lambda(j)} , \bv_k(x)\right\ra\right)_{j,k}\right|  \\
\leq (n-m) \left( \max_{j=1,\ldots,n-m} \left| h_{\lambda(j)}\right|\right)\left( \max_{j=1,\ldots,n-m} \left| h_{\lambda(j)}\right| + \left| e_{\lambda(j)}\right|\right)^{n-m-1} \\
\leq (n-m)m\Lambda|x-u| \left( 1 + m \Lambda |x-u|\right)^{n-m-1} .
\end{multline*}
This completes the proof of conclusion (1).
\par 
Let $E_0$ denote the subset of $E$ consisting of those $x$ such that each $\bw_i$, $i=1,\ldots,m$, is differentiable at $x$. 
Thus, $E_0$ is Borel measurable and so is
\begin{multline*}
A = E_0 \times \Rm \cap \bigg\{ (x,t) : \\ \rmrank 
\left(
\begin{array}{c|c|c|c|c|c}
\bw_1(x) & \cdots & \bw_m(x) & e_1 + \sum_{i=1}^m t_i \frac{\partial \bw_i}{\partial x_1}(x) & \ldots & e_n + \sum_{i=1}^m t_i \frac{\partial \bw_i}{\partial x_n}(x) 
\end{array}
\right) < n
\bigg\}.
\end{multline*}
If $(x,u) \in \Sigma \setminus F(A)$, then the restriction of $\pi_2$ to $T_{(x,u)}\Sigma$ is surjective and, therefore, $J_\Sigma \pi_2(x,u) > 0$. 
Thus, we ought to show that $\calH^{n+m}(F(A))=0$. 
Since $F$ is Lipschitzian, it suffices to establish that $\calL^{n+m}(A)=0$. 
As $A$ is Borel measurable, it is enough to prove that $\calL^m(A_x)=0$, for every $x \in E_0$, according to Fubini's theorem. 
Fix $x \in E_0$. 
As in the proof of conclusion (1), choose $\lambda \in \Lambda(n,n-m)$ associated with $\bv_1(x),\ldots,\bv_{n-m}(x)$, according to \ref{34}. 
Based on \eqref{eq.13}, we see that
\begin{equation*}
A_x \subset \Rm \cap \left\{ t :  \det \left( \left\la e_{\lambda(j)} + \sum_{i=1}^m t_i \frac{\partial \bw_i}{\partial x_{\lambda(j)}}(x) , \bv_k(x) \right\ra\right)_{j,k=1,\ldots,n-m}=0\right\}.
\end{equation*}
The set on the right is of the form $S_x = \Rm \cap \{ (t_1,\ldots,t_m ) : P_x(t_1,\ldots,t_m) = 0 \}$, for some polynomial $P_x \in \R[T_1,\ldots,T_m]$, and $P_x(0,\ldots,0) = \det\left( \la e_{\lambda(j)},\bv_k(x) \ra\right)_{j,k=1,\ldots,n-m} \neq 0$. 
It follows that $\calL^m(S_x)=0$ -- see, e.g., \cite[2.6.5]{GMT} -- and the proof of (2) is complete.
\end{proof}

\begin{Proposition}
\label{AC.1}
The measure $\phi_{E,\bW}$ is absolutely continuous with respect to $\calL^n$.
\end{Proposition}

\begin{proof}
Let $B \in \calB(\Rn)$ be such that $\calL^n(B)=0$. 
It follows from \eqref{eq.10} that
\begin{equation*}
\int_{\Sigma_B} J_\Sigma \pi_2 d\calH^{n+m} = 0 .
\end{equation*}
It next follows from \ref{factor.pi.2}(2) that $\calH^{n+m}(\Sigma_B)=0$. 
In turn, \eqref{eq.9} implies that
\begin{equation*}
\phi_{E,\bW}(B) = \int_{\Sigma_B} J_\Sigma \pi_1 d\calH^{n+m} = 0 .
\end{equation*}
\end{proof}

\begin{Empty}[Definition of $\calZ_E \bW$]
\label{def.Z}
Note that $\phi_{E,\bW}$ is a $\sigma$-finite Borel measure on $\Rn$ (see \ref{def.phi}) and it is absolutely continuous with respect to $\calL^n$ (see \ref{AC.1}). 
It then ensues from the Radon-Nikod\'ym theorem that there exists a Borel measurable function
\begin{equation*}
\calZ_E \bW : \Rn \to \R
\end{equation*}
such that, for every $B \in \calB(\Rn)$, one has
\begin{equation*}
\int_E \calH^m \left( B \cap \bW(x) \right) d\calL^n(x) = \phi_{E,\bW}(B) = \int_B \calZ_E \bW(u) d\calL^n(u) .
\end{equation*}
Furthermore, $\calZ_E \bW$ is univoquely defined (only) up to an $\calL^n$ null set. 
This will not affect the reasoning in this paper. 
Each time we write $\calZ_E \bW$, we mean {\it one} particular Borel measurable function satisfying the above equality, for every $B \in \calB(\Rn)$.
\end{Empty}

\begin{Empty}[Definition of $\calY_E^0 \bW$]
\label{def.Y0}
We define $\calY^0_E\bW : \Rn \to [0,\infty]$ by the formula
\begin{equation}
\label{eq.14}
\calY_E^0 \bW(u) = \calH^m \left( E \cap g_{\bv_1,\ldots,\bv_{n-m},u}^{-1}\{0\}\right) ,
\end{equation}
$u \in \Rn$. 
Letting $B = \Rn$ in \eqref{eq.8}, one infers from \ref{coarea} that $\calY^0_E\bW$ is $\calL^n$-measurable. 
Using the estimates we have established so far regarding coarea Jacobian factors, we now show that $\calZ_E \bW$ and $\calY_E^0 \bW$ are comparable, when the diameter of $E$ is not too large.
\end{Empty}

\begin{Proposition}
\label{Z.1}
Given $0 < \veps < 1$, there exists $\bdelta_{\theTheorem}(n,\Lambda,\veps) > 0$ with the following property. If $\rmdiam E \leq \bdelta_{\theTheorem}(n,\Lambda,\veps)$, then
\begin{equation*}
(1-\veps) 2^{-\frac{n-m}{2}} \calY^0_E \bW(u) \leq \calZ_E \bW(u) \leq (1+\veps) 2^\frac{n-m}{2} \bin{n}{n-m}^\frac{1}{2	}\calY^0_E \bW(u),
\end{equation*}
for $\calL^n$-almost every $u \in E$.
\end{Proposition}

\begin{proof}
We readily infer from \ref{factor.pi.1} and \ref{factor.pi.2}(1) that there exists $\bdelta(n,\Lambda,\veps) > 0$ such that, for $\calH^{n+m}$-almost all $(x,u) \in \Sigma$, if $|x-u| \leq \bdelta(n,\Lambda,\veps)$, then 
\begin{equation}
\alpha := (1-\veps) 2^{-\frac{n-m}{2}} \leq J_\Sigma \pi_1(x,u)
\end{equation}
and
\begin{equation}
\beta := (1+\veps)^{-1} 2^{-\frac{n-m}{2}} \bin{n}{n-m}^{-\frac{1}{2}} \leq J_\Sigma \pi_2(x,u) ,
\end{equation}
where the above define $\alpha$ and $\beta$.
\par 
Assume now that $\rmdiam E \leq \bdelta(n,\Lambda,\veps)$. 
Given $B \in \calB(E)$, we infer from \eqref{eq.9}, \ref{factor.pi.1}, \ref{factor.pi.2}(1), \eqref{eq.10}, and the above lower bounds, that
\begin{multline*}
\phi_{E,\bW}(B)  = \int_{\Sigma_B} J_\Sigma \pi_1 d\calH^{n+m} 
 \geq \alpha \calH^{n+m}(\Sigma_B) 
 \geq \alpha \int_{\Sigma_B} J_\Sigma \pi_2 d\calH^{n+m} \\
 = \alpha \int_B \calY^0_E \bW d\calL^n
\end{multline*}
and
\begin{multline*}
\phi_{E,\bW}(B) = \int_{\Sigma_B} J_\Sigma \pi_1 d\calH^{n+m} \leq \calH^{n+m}(\Sigma_B) \leq \beta^{-1} \int_{\Sigma_B} J_\Sigma \pi_2 d\calH^{n+m} \\
 = \beta^{-1} \int_B  \calY^0_E \bW d\calL^n .
\end{multline*}
Thus,
\begin{equation*}
\int_B  \alpha\calY^0_E \bW d\calL^n \leq \int_B \calZ_E \bW d\calL^n \leq \int_B \beta^{-1} \calY^0_E \bW d\calL^n,
\end{equation*}
for every $B \in \calB(E)$. 
The conclusion follows from the $\calL^n$-measurability of both $\calZ_E \bW$ and $\calY^0_E \bW$.
\end{proof}

\begin{Empty}[Rest stop]
The above upper bound for $\calZ_E \bW$ is already enough to bound it from above, in turn, by a constant times $(\rmdiam E)^m$ -- see \ref{cor.ub}. 
We next want to establish that $\calZ_E \bW > 0$, almost everywhere in $E$.
Yet, in the definition \eqref{eq.14} of $\calY^0_E \bW(u)$, $u$ does not appear as the covariable of a function whose level sets we are measuring, thereby preventing the use of the coarea formula in an attempt to estimate $\calY^0_E \bW(u)$. 
This naturally leads to adding a variable $y \in \R^{n-m}$ to the fibered space $\Sigma$, a covariable for $g_{\bv_1,\ldots,\bv_{n-m},u}$.
\end{Empty}

\begin{Empty}[A fibered space associated with $E, B, \bw_1,\ldots,\bw_m,\bv_1,\ldots,\bv_{n-m}$]
\label{fibration.2}
Let $r > 0$.
Abbreviate $C_r = \R^{n-m} \cap \left\{ y : |y| \leq r \right\}$, the Euclidean ball centered at the origin, of radius $r$, in $\R^{n-m}$. 
We define
\begin{equation*}
\hat{F}_r : E \times \Rm \times C_r \to \Rn \times \Rn \times \R^{n-m} : (x,t,y) \mapsto \left(x , x + \sum_{i=1}^m t_i \bw_i(x) + \sum_{i=1}^{n-m} y_i \bv_i(x) , y \right)
\end{equation*}
and
\begin{equation*}
\hat{\Sigma}_r = \hat{F}_r \left( E \times \Rm \times C_r \right) = \Rn \times \Rn \times C_r \cap \left\{ (x,u,y) : x \in E \text{ and } u \in \bW(x) + \sum_{i=1}^{n-m} y_i\bv_i(x)\right\}.
\end{equation*}
We note that $\hat{F}_r$ is locally Lipschitzian and $\hat{\Sigma}_r$ is countably $2n$-rectifiable and $\calH^{2n}$-measurable. 
Similarly to \ref{fibration.1}, we define
\begin{equation*}
\hat{\Sigma}_{r,B} = \hat{\Sigma}_r \cap \pi_2^{-1}(B) 
\end{equation*}
which, clearly, is also countably $2n$-rectifiable and $\calH^{2n}$-measurable. 
We aim to apply the coarea formula to $\hat{\Sigma}_{r,B}$ and to the two projections
\begin{equation*}
\pi_1 \times \pi_3 : \Rn \times \Rn \times \R^{n-m} \to \Rn \times \R^{n-m} : (x,u,y) \mapsto (x,y)
\end{equation*}
and
\begin{equation*}
\pi_2 \times \pi_3 : \Rn \times \Rn \times \R^{n-m} \to \Rn \times \R^{n-m} : (x,u,y) \mapsto (u,y) .
\end{equation*}
To this end, we notice that
\begin{multline*}
\hat{\Sigma}_{r,B} \cap \left( \pi_1 \times \pi_3 \right)^{-1} \{ (x,y) \} \\ = \Rn \times \Rn \times \R^{n-m} \cap \left\{ (x,u,y) : u \in B \cap \left( \bW(x) + \sum_{i=1}^{n-m} y_i \bv_i(x) \right)\right\}
\end{multline*}
and, thus,
\begin{equation*}
\calH^m \left( \hat{\Sigma}_{r,B} \cap \left( \pi_1 \times \pi_3 \right)^{-1} \{ (x,y) \}\right) = \calH^m \left(B \cap \left( \bW(x) + \sum_{i=1}^{n-m} y_i \bv_i(x) \right) \right),
\end{equation*}
for every $(x,y) \in E \times C_r$. 
We further notice that
\begin{equation*}
\begin{split}
\hat{\Sigma}_{r,B} \cap ( \pi_2 & \times \pi_3 )^{-1} \{ (u,y) \} \\ &= \Rn \times \Rn \times \R^{n-m} \cap \left\{ (x,u,y) : x \in E \text{ and } u \in \bW(x) + \sum_{i=1}^{n-m} y_i\bv_i(x)\right\}\\
&= \Rn \times \Rn \times \R^{n-m} \cap \left\{ (x,u,y) : x \in E \cap g_{\bv_1,\ldots,\bv_{n-m},u}^{-1}\{y\} \right\} ,
\end{split}
\end{equation*}
because
\begin{equation*}
\begin{split}
u \in \bW(x) + \sum_{i=1}^{n-m} y_i \bv_i(x) & \iff u-x- \sum_{i=1}^{n-m} y_i \bv_i(x) \in \bW_0(x) \\
& \iff \left\la \bv_j(x) , u-x- \sum_{i=1}^{n-m} y_i \bv_i(x) \right\ra = 0 \text{, for all } j=1,\ldots,n-m \\
& \iff g_{\bv_1,\ldots,\bv_{n-m},u}(x) = y
\end{split}
\end{equation*}
and, therefore,
\begin{equation*}
\calH^m \left(\hat{\Sigma}_{r,B} \cap ( \pi_2 \times \pi_3 )^{-1} \{ (u,y) \} \right) = \calH^m \left(E \cap g_{\bv_1,\ldots,\bv_{n-m},u}^{-1}\{y\} \right),
\end{equation*}
whenever $u \in B$ and $y \in C_r$.
\par 
It now follows from the coarea formula and Fubini's theorem that
\begin{equation}
\label{eq.20}
\int_{\hat{\Sigma}_{r,B}} J_{\hat{\Sigma}_r} (\pi_1 \times \pi_3) d\calH^{2n} = \int_E  d\calL^n(x)\int_{C_r} \calH^m \left(B \cap \left( \bW(x) + \sum_{i=1}^{n-m} y_i \bv_i(x) \right) \right) d\calL^{n-m}(y) 
\end{equation}
and
\begin{equation}
\label{eq.21}
\int_{\hat{\Sigma}_{r,B}} J_{\hat{\Sigma}_r} (\pi_2 \times \pi_3) d\calH^{2n} = \int_B d\calL^n(u) \int_{C_r} \calH^m \left( E \cap g_{\bv_1,\ldots,\bv_{n-m},u}^{-1}\{y\} \right)d\calL^{n-m}(y) .
\end{equation}
\end{Empty}

\begin{Empty}[Coarea Jacobian factors of $\pi_1 \times \pi_3$ and $\pi_2 \times \pi_3$]
\label{factor}
{\it 
The following inequalities hold, for $\calH^{2n}$-almost every $(x,u,y) \in \hat{\Sigma}_r$.
\begin{equation*}
2^{-(n-m)} \left( 1 + 2n\Lambda |u-x| + 2n^2 \Lambda^2 |u-x|^2 \right)^{-\frac{n}{2}}
 \leq J_{\hat{\Sigma}_r} (\pi_1 \times \pi_3)(x,u,y)
\end{equation*}
and 
\begin{equation*}
 J_{\hat{\Sigma}_r} (\pi_2 \times \pi_3)(x,u,y) \leq 1 .
\end{equation*}
}
\end{Empty}

\begin{proof}
The second conclusion is obvious, since $\rmLip \pi_2 \times \pi_3 = 1$. 
Regarding the first conclusion, we reason similarly to the proof of \ref{factor.pi.1}. 
Fix $(x,u,y) = \hat{F}_r(x,t,y)$ such that $\hat{F}_r$ is differentiable at $(x,t,y)$ and denote by $L$ the restriction of $\pi_1 \times \pi_3$ to $T_{(x,u,y)}\hat{\Sigma}_r$. 
This tangent space is generated by the following $2n$ vectors of $\Rn \times \Rn \times \R^{n-m}$:
\begin{equation*}
\begin{split}
\frac{\partial \hat{F}_r}{\partial x_j}(x,t,y) & = \left( e_j , e_j + \sum_{i=1}^m t_i \frac{\partial \bw_i(x)}{\partial x_j}(x) + \sum_{i=1}^{n-m} y_i \frac{\partial \bv_i}{\partial x_j}(x) , 0 \right) ,\, j=1,\ldots,n, \\
\frac{\partial \hat{F}_r}{\partial t_k}(x,t,y) & = \left( 0, \bw_k(x), 0 \right) ,\, k=1,\ldots,m, \\
\frac{\partial \hat{F}_r}{\partial y_\ell}(x,t,y) & = \left( 0, \bv_\ell(x) , e_\ell \right) ,\, \ell=1,\ldots,n-m .
\end{split}
\end{equation*}
The range of $\pi_1 \times \pi_3$ being $(2n -m)$-dimensional, we need to select $2n-m$ vectors $v_1,\ldots,v_{2n-m}$ in $ T_{(x,u,y)}\hat{\Sigma}_r$ in view of obtaining a lower bound
\begin{equation}
\label{eq.22}
J_{\hat{\Sigma}_r} (\pi_1 \times \pi_3)(x,u,y) = \| \wedge_{2n-m} L \| \geq \frac{|L(v_1) \wedge \ldots \wedge L(v_{2n-m})|}{|v_1 \wedge \ldots \wedge v_{2n-m}|} .
\end{equation}
Our choice of $v_1,\ldots,v_{2n-m}$ is as follows.
As in the proof of \ref{factor.pi.1}, we let $w_j=\bw_j(x)$, for $j=1,\ldots,m$, and $w_j=\bv_{j-m}(x)$, for $j=m+1,\ldots,n$.
For $j=1,\ldots,m$, we define
\begin{multline*}
v_j = \left(\sum_{p=1}^n \la w_j , e_p \ra \frac{\partial \hat{F}_r}{\partial x_p}(x,t,y)\right) - \frac{\partial \hat{F}_r}{\partial t_j}(x,t,y) \\
= \left( w_j , \sum_{i=1}^m t_i D\bw_i(x)(w_j) + \sum_{i=1}^{n-m} y_i D\bv_i(x)(w_j) , 0 \right) ,
\end{multline*}
for $j=m+1,\ldots,n$, we define
\begin{equation*}
v_j = \sum_{p=1}^m \la w_j , e_p \ra \frac{\partial \hat{F}_r}{\partial x_p}(x,t,y)
= \left( w_j , w_j + \sum_{i=1}^m t_i D\bw_i(x)(w_j) + \sum_{i=1}^{n-m} y_i D\bv_i(x)(w_j) , 0 \right) ,
\end{equation*}
and, for $j=n+1,\ldots,2n-m$, we define
\begin{equation*}
v_j = \frac{\partial \hat{F}_r}{\partial y_{j-n}}(x,t,y) = ( 0 , \bv_{j-n}(x),e_{j-n}) ,
\end{equation*}
so that $L(v_1),\ldots,L(v_{2n-m})$ is an orthonormal basis of $\Rn \times \R^{n-m}$ and, therefore, the numerator in \eqref{eq.22} equals 1. 
In order to determine an upper bound for its denominator, we start by fixing $j=1,\ldots,n$, we abbreviate $a_j(x,t,y) = \sum_{i=1}^m t_i D\bw_i(x)(x)(w_j)$ and $b_j(x,t,y) = \sum_{i=1}^{n-m} y_i D\bv_i(x)(w_j)$, and we notice that $|a_j(x,t,y)| \leq m \Lambda |t| \leq n \Lambda |t|$, $|b_j(x,t,y)| \leq (n-m) \Lambda |y| \leq n \Lambda |y|$. 
Furthermore, since $u - x = \sum_{i=1}^m t_i \bw_i(x) + \sum_{i=1}^{n-m} y_i \bv_i(x)$, one has $|u-x|^2 = |t|^2 + |y|^2 \geq \max \{|t|^2,|y|^2\}$. 
Therefore, if $j=1,\ldots,m$, then
\begin{equation*}
\begin{split}
|v_j|^2 & = |w_j|^2 + \left| a_j(x,t,y) + b_j(x,t,y) \right|^2 \\
& \leq 1 + \left|a_j(x,t,y)\right|^2 + \left|b_j(x,t,y)\right|^2 + 2 \left|a_j(x,t,y)\right| \left|b_j(x,t,y)\right| \\ 
& \leq 1 + n^2 \Lambda^2 \left( |t|^2 + |y|^2 + 2 |t| |y|\right) \\
& \leq 1 + 2 n^2 \Lambda^2 |u-x|^2,
\end{split}
\end{equation*}
whereas, if $j=m+1,\ldots,n$, then
\begin{equation*}
\begin{split}
\left| v_j\right|^2 &= |w_j|^2 + |w_j + a_j(x,t,y) + b_j(x,t,y)|^2 \\
& \leq 1 + 1 + |a_j(x,t,y)|^2 + |b_j(x,t,y)|^2 \\
& \quad \quad \quad+ 2|a_j(x,t,y)| + 2 |b_j(x,t,y)| + 2 |a_j(x,t,y)||b_j(x,t,y)| \\
& \leq 2 \left( 1 + n^2 \Lambda^2 |u-x|^2 + 2n\Lambda |u-x| \right) ,
\end{split}
\end{equation*}
and, if $j=n+1,\ldots,2n-m$, then
\begin{equation*}
\left| v_j \right| = \sqrt{2} .
\end{equation*}
We conclude that
\begin{multline*}
|v_1\wedge \ldots \wedge v_{2n-m} | 
\leq 2^{n-m}\left( 1 + 2n^2 \Lambda^2|u-x|^2\right)^\frac{m}{2}\left( 1 + n^2 \Lambda^2 |u-x|^2+ 2n\Lambda |u-x|  \right)^\frac{n-m}{2} \\
\leq  2^{n-m}\left( 1 + 2n^2 \Lambda^2|u-x|^2+2n\Lambda |u-x|  \right)^\frac{n}{2}
\end{multline*}
and the proof is complete.
\end{proof}

\begin{Empty}[Definition of $\calY_E \bW$]
\label{def.Y}
It follows from the coarea theorem that the function
\begin{equation*}
\Rn \times \R^{n-m} \to [0,\infty] : (u,y) \to \calH^m \left( E \cap g_{\bv_1,\ldots,\bv_{n-m},u}^{-1} \{y\}\right) 
\end{equation*}
is $\calL^n \otimes \calL^{n-m}$-measurable (recall \ref{fibration.2} applied with $B = \Rn$). 
It now follows from Fubini's theorem that, for each $r > 0$, the function
\begin{equation*}
\Rn \to [0,\infty] : u \mapsto \dashint_{C_r} \calH^m \left( E \cap g_{\bv_1,\ldots,\bv_{n-m},u}^{-1} \{y\}\right) d\calL^{n-m}(y)
\end{equation*}
is $\calL^n$-measurable. 
In turn, the function
\begin{equation*}
\calY_E \bW : \Rn \to [0,\infty] : u \mapsto \liminf_j \dashint_{C_{j^{-1}}} \calH^m \left( E \cap g_{\bv_1,\ldots,\bv_{n-m},u}^{-1} \{y\}\right) d\calL^{n-m}(y)
\end{equation*}
is $\calL^n$-measurable. 
It is a replacement for $\calY^0_E \bW$ defined in \ref{def.Y0}. 
We shall establish, for $\calZ_E \bW$, a similar lower bound to that in \ref{Z.1}, this time involving $\calY_E \bW$. 
Before doing so, we notice the rather trivial fact that if $F \subset E$, then
\begin{equation*}
\calY_F \bW(u) \leq \calY_E \bW(u),
\end{equation*}
for all $u \in \Rn$.
\end{Empty}

\begin{Empty}[Preparatory remark for the proof of \ref{lb.1}]
\label{rem.1}
It follows from the coarea theorem that the function
\begin{equation*}
\Rn \times \R^{n-m} \to [0,\infty] : (x,y) \mapsto \calH^m \left(B \cap \left( \bW(x) + \sum_{i=1}^{n-m} y_i \bv_i(x) \right) \right) 
\end{equation*}
is $\calL^n \otimes \calL^{n-m}$-measurable (recall \ref{fibration.2} applied with $B = \Rn$). 
It therefore follows from Fubini's theorem as in \ref{def.Y} that
\begin{equation*}
f_j : \Rn \to [0,\infty] : x \mapsto \dashint_{C_{j^{-1}}} \calH^m \left(B \cap \left( \bW(x) + \sum_{i=1}^{n-m} y_i \bv_i(x) \right) \right) d\calL^{n-m}(y)
\end{equation*}
is $\calL^n$-measurable. 
Furthermore, if $B$ is bounded, then $|f_j(x)| \leq \balpha(m) (\rmdiam B)^m$, for every $x \in \Rn$. 
\end{Empty}

\begin{Empty}
\label{usc}
{\it 
If $B$ is compact, then, for every $x \in E,$ the function
\begin{equation*}
\R^{n-m} \to \R_+ : y \mapsto \calH^m \left(B \cap \left( \bW(x) + \sum_{i=1}^{n-m} y_i \bv_i(x) \right) \right) 
\end{equation*}
is upper semicontinuous.
}
\end{Empty}

\begin{proof}
The proof is analogous to that of \ref{measurability}.
For each $y \in \R^{n-m}$, define the compact set $K_y = B \cap \left( \bW(x) + \sum_{i=1}^{n-m} y_i \bv_i(x) \right)$. 
If $(y_k)_k$ is a sequence converging to $y$, we ought to show that
\begin{equation*}
\calH^m_\infty \left(K_y\right) \geq \limsup_k \calH^m_\infty \left( K_{y_k} \right).
\end{equation*}
Since each $K_y$ is a subset of an $m$-dimensional affine subspace of $\Rn$, this is indeed equivalent to the same inequality with $\calH^m_\infty$ replaced by $\calH^m$, according to \ref{h.m}(3). 
Considering, if necessary, a subsequence of $(y_k)_k$, we may assume that none of the compact sets $K_{y_k}$ is empty and that the above limit superior is a limit. 
Considering yet a further subsequence, we may now assume that $(K_{y_k})_k$ converges in Hausdorff distance to some compact set $L \subset B$. 
One checks that $L \subset K_y$. 
It then follows from \ref{h.m}(1) that $\calH^m_\infty \left(K_y\right) \geq \calH^m_\infty (L) \geq \limsup_k \calH^m_\infty \left( K_{y_k} \right)$.
\end{proof}

\begin{Proposition}
\label{lb.1}
Given $0 < \veps < 1$, there exists $\bdelta_{\theTheorem}(n,\Lambda,\veps) > 0$ with the following property.
If $\rmdiam ( E \cup B) \leq \bdelta_{\theTheorem}(n,\Lambda,\veps)$ and $B$ is compact, then
\begin{equation*}
\int_E \calH^m \left( B \cap \bW(x) \right) d\calL^n(x) \geq (1-\veps)2^{-(n-m)} \int_B \calY_E \bW(u) d\calL^n(u) .
\end{equation*}
\end{Proposition}

\begin{proof}
We first observe that we can choose $\bdelta_{\ref{lb.1}}(n,\Lambda,\veps) > 0$ small enough so that 
\begin{equation}
\label{eq.23}
J_{\hat{\Sigma}_r} (\pi_1 \times \pi_3)(x,u,y) \geq (1-\veps) 2^{-(n-m)},
\end{equation}
for $\calH^{2n}$-almost every $(x,u,y) \in \hat{\Sigma}_r$, provided $|u-x| \leq \bdelta_{\ref{lb.1}}(n,\Lambda,\veps)$, according to \ref{factor}. 
Thus, \eqref{eq.23} holds, for $\calH^{2n}$-almost every $(x,u,y) \in \hat{\Sigma}_{r,B}$, under the assumption that  $\rmdiam ( E \cup B) \leq \bdelta_{\ref{lb.1}}(n,\Lambda,\veps)$. 
In that case, \eqref{eq.20}, \eqref{eq.21}, and \ref{factor} imply that
\begin{equation}
\label{eq.25}
\begin{split}
\int_E  d\calL^n(x) & \int_{C_r} \calH^m \left(B \cap \left( \bW(x) + \sum_{i=1}^{n-m} y_i \bv_i(x) \right) \right) d\calL^{n-m}(y) \\
&= \int_{\hat{\Sigma}_{r,B}} J_{\hat{\Sigma}_r} (\pi_1 \times \pi_3) d\calH^{2n} \\
& \geq (1-\veps)2^{-(n-m)} \calH^{2n} \left( \hat{\Sigma}_{r,B}\right) \\
& \geq (1-\veps) 2^{-(n-m)}\int_{\hat{\Sigma}_{r,B}} J_{\hat{\Sigma}_r} (\pi_2 \times \pi_3) d\calH^{2n} \\
& = (1-\veps)  2^{-(n-m)} \int_B d\calL^n(u) \int_{C_r} \calH^m \left( E \cap g_{\bv_1,\ldots,\bv_{n-m},u}^{-1}\{y\} \right)d\calL^{n-m}(y) .
\end{split}
\end{equation}
\par 
Fix $x \in E$ and $\beta > 0$. 
According to \ref{usc}, there exists a positive integer $j(x,\beta)$ such that if $j \geq j(x,\beta)$, then
\begin{equation*}
\begin{split}
\calH^m \left( B \cap \bW(x) \right) + \beta & \geq \sup_{y \in C_{j^{-1}}}
\calH^m \left(B \cap \left( \bW(x) + \sum_{i=1}^{n-m} y_i \bv_i(x) \right) \right) \\
& \geq \dashint_{C_{j^{-1}}}\calH^m \left(B \cap \left( \bW(x) + \sum_{i=1}^{n-m} y_i \bv_i(x) \right) \right) d\calL^{n-m}(y) .
\end{split}
\end{equation*} 
Taking the limit superior as $j \to \infty$, on the right hand side, and letting $\beta \to 0$, we obtain
\begin{equation}
\label{eq.24}
\calH^m \left( B \cap \bW(x) \right) \geq \limsup_j \dashint_{C_{j^{-1}}}\calH^m \left(B \cap \left( \bW(x) + \sum_{i=1}^{n-m} y_i \bv_i(x) \right) \right) d\calL^{n-m}(y) .
\end{equation}
As this holds for all $x \in E$, we may integrate over $E$ with respect to $\calL^n$. 
We notice that for every $j=1,2,\ldots$, $|f_j| \leq \balpha(m) (\rmdiam B)^m \ind_{\Rn}$ (recall the notation of \ref{rem.1}); the latter being $\calL^n \hel E$-summable, this justifies the application of the reverse Fatou lemma below. 
Thus, the following ensues from \eqref{eq.24}, the reverse Fatou lemma, \eqref{eq.25}, and the Fatou lemma:
\begin{equation*}
\begin{split}
\int_E \calH^m &\left( B \cap \bW(x) \right) d\calL^n(x) \\& \geq \int_E d\calL^n(x) \limsup_j \dashint_{C_{j^{-1}}}\calH^m \left(B \cap \left( \bW(x) + \sum_{i=1}^{n-m} y_i \bv_i(x) \right) \right) d\calL^{n-m}(y) \\
& \geq \limsup_j \int_E d\calL^n(x)\dashint_{C_{j^{-1}}}\calH^m \left(B \cap \left( \bW(x) + \sum_{i=1}^{n-m} y_i \bv_i(x) \right) \right) d\calL^{n-m}(y) \\
& \geq (1-\veps) 2^{-(n-m)} \limsup_j \int_B d\calL^n(u) \dashint_{C_{j^{-1}}} \calH^m \left( E \cap g_{\bv_1,\ldots,\bv_{n-m},u}^{-1}\{y\} \right)d\calL^{n-m}(y)\\
& \geq (1-\veps)  2^{-(n-m)} \liminf_j \int_B d\calL^n(u) \dashint_{C_{j^{-1}}} \calH^m \left( E \cap g_{\bv_1,\ldots,\bv_{n-m},u}^{-1}\{y\} \right)d\calL^{n-m}(y)\\
& \geq (1-\veps)  2^{-(n-m)}  \int_B d\calL^n(u)\liminf_j \dashint_{C_{j^{-1}}} \calH^m \left( E \cap g_{\bv_1,\ldots,\bv_{n-m},u}^{-1}\{y\} \right)d\calL^{n-m}(y)\\
&= (1-\veps)  2^{-(n-m)}\int_B \calY_E \bW(u)d\calL^n(u) .
\end{split}
\end{equation*}
\end{proof}

\begin{Corollary}
\label{lb.2}
If $0 < \veps < 1$ and $\rmdiam E \leq \bdelta_{\ref{lb.1}}(n,\Lambda,\veps)$, then 
\begin{equation*}
\calZ_E \bW(u) \geq (1-\veps)  2^{-(n-m)}\calY_E \bW(u),
\end{equation*}
for $\calL^n$-almost every $u \in E$.
\end{Corollary}

\section{Upper bound for $\calY_E \bW$ and $\calZ_E \bW$}

\begin{Empty}[Bow tie lemma]
\label{bow.tie}
{\it 
Let $S \subset \Rn$, $W \in \bG(n,m)$, and $0 < \tau < 1$. Assume that 
\begin{equation*}
(\forall \,x \in S)(\forall \,0 < \rho \leq \rmdiam S) : S \cap \bB(x,\rho) \subset \bB(x+W,\tau\rho) .
\end{equation*}
Then there exists $F : P_W(S) \to \Rn$ such that $S=\rmim F$ and $\rmLip F \leq \frac{1}{\sqrt{1-\tau^2}}$. In particular,
\begin{equation*}
\calH^m(S) \leq \left(\frac{1}{\sqrt{1-\tau^2}} \right)^m \balpha(m) (\rmdiam S)^m .
\end{equation*}
}
\end{Empty}

\begin{proof}
Let $x,x' \in S$ and define $\rho = |x-x'| \leq \rmdiam S$. 
Thus, $x' \in S \cap \bB(x,\rho)$ and, therefore, $\left|P_{W^\perp}(x-x') \right| \leq \tau \rho = \tau |x-x'|$. 
Since $|x-x'|^2 = \left|P_{W}(x-x') \right|^2 + \left|P_{W^\perp}(x-x') \right|^2$, we infer that
\begin{equation*}
(1-\tau^2) \left|x-x'\right|^2 \leq \left|P_{W}(x-x') \right|^2.
\end{equation*}
Therefore, $P_W|_S$ is injective; and the Lipschitzian bound on $F = \left( P_W|_S \right)^{-1}$ clearly follows from the above inequality. 
Regarding the second conclusion, we note that
\begin{equation*}
\calH^m(S) = \calH^m\left(F(P_W(S)) \right) \leq \left( \rmLip F\right)^m \calH^m \left(P_W(S) \right)
\end{equation*}
and $P_W(S)$ is contained in a ball of radius $\rmdiam P_W(S) \leq \rmdiam S$.
\end{proof}

\begin{Empty}
\label{ub.1}
{\it 
Given $0 < \tau < 1$, there exists $\bdelta_{\theTheorem}(n,\Lambda,\tau) > 0$ with the following property. If 
\begin{enumerate}
\item $x_0 \in U$ and $u \in \Rn$;
\item $\rmdiam\left( E \cup \{x_0\}\cup\{u\}\right) \leq \bdelta_{\theTheorem}(n,\Lambda,\tau)$;
\end{enumerate}
then: For each $y \in \R^{n-m}$, each $x \in E \cap g_{\bv_1,\ldots,\bv_{n-m},u}^{-1}\{y\}$, and each $0 < \rho < \infty$, one has
\begin{equation*}
E \cap g_{\bv_1,\ldots,\bv_{n-m},u}^{-1}\{y\} \cap \bB(x,\rho) \subset \bB \left( x + \bW_0(x_0),\tau \rho\right).
\end{equation*}
}
\end{Empty}

\begin{proof}
We show that $\bdelta_{\ref{ub.1}}(n,\Lambda,\tau)=\frac{\tau}{2\Lambda \sqrt{n}}$ will do.
Let $x,x' \in E \cap g_{\bv_1,\ldots,\bv_{n-m},u}^{-1}\{y\}$, for some $y \in \R^{n-m}$. 
Thus, $g_{\bv_1,\ldots,\bv_{n-m},u}(x)=g_{\bv_1,\ldots,\bv_{n-m},u}(x')$ and, hence,
\begin{multline*}
0 = \left| g_{\bv_1,\ldots,\bv_{n-m},u}(x) - g_{\bv_1,\ldots,\bv_{n-m},u}(x')\right| = \sqrt{\sum_{i=1}^{n-m} \left| \la \bv_i(x),x-u\ra -  \la \bv_i(x'),x'-u\ra\right|^2} \\
= \sqrt{\sum_{i=1}^{n-m} \left| \la \bv_i(x),x-x'\ra - \la \bv_i(x') - \bv_i(x),x'-u \ra\right|^2}\\
\geq \sqrt{\sum_{i=1}^{n-m} \left| \la \bv_i(x),x-x'\ra\right|^2} - \sqrt{\sum_{i=1}^{n-m}\left| \la \bv_i(x') - \bv_i(x),x'-u \ra\right|^2},
\end{multline*}
thus,
\begin{multline*}
\sqrt{\sum_{i=1}^{n-m} \left| \la \bv_i(x),x-x'\ra\right|^2} \leq \sqrt{\sum_{i=1}^{n-m}\left| \la \bv_i(x') - \bv_i(x),x'-u \ra\right|^2} \\
\leq \sqrt{n-m} \Lambda |x-x'||x'-u| \leq \frac{\tau}{2}|x-x'|.
\end{multline*}
In turn,
\begin{multline*}
\left| P_{\bW_0(x_0)^\perp}(x-x')\right| = \sqrt{\sum_{i=1}^{n-m}\left|\la \bv_i(x_0),x-x' \ra \right|^2} \\ \leq \sqrt{\sum_{i=1}^{n-m}\left|\la \bv_i(x'),x-x' \ra \right|^2}+\sqrt{\sum_{i=1}^{n-m}\left|\la \bv_i(x') - \bv_i(x_0),x-x' \ra \right|^2}\\
\leq \frac{\tau}{2}|x-x'| + \sqrt{n-m}\Lambda |x'-x_0||x-x'| \leq \tau |x-x'|.
\end{multline*}
\end{proof}

\begin{Proposition}
\label{upper.bound}
There are $\bdelta_{\theTheorem}(n,\Lambda) > 0$ and $\bc_{\theTheorem}(m) \geq 1$ with the following property. If $u \in \Rn$ and $\rmdiam (E \cup \{u\}) \leq \bdelta_{\theTheorem}(n,\Lambda)$, then
\begin{equation*}
\max \left\{ \calY^0_E \bW(u) , \calY_E \bW(u)\right\} \leq \bc_{\theTheorem}(m) (\rmdiam E)^m .
\end{equation*}
\end{Proposition}

\begin{proof}
Let $\bdelta_{\ref{upper.bound}}(n,\Lambda) =\bdelta_{\ref{ub.1}}(n,\Lambda,1/2)$.
Recall the definitions of $\calY^0_E \bW$ and $\calY_E \bW$, \ref{def.Y0} and \ref{def.Y}, respectively. 
If $E =\emptyset$, the conclusion is obvious. 
If not, pick $x_0 \in E$ arbitrarily. 
Given any $y \in \R^{n-m}$, we see that \ref{ub.1} applies with $\tau = 1/2$ and, in turn, the bow tie lemma \ref{bow.tie} applies to $S = E \cap g_{\bv_1,\ldots,\bv_{n-m},u}^{-1}\{y\}$ and $W = \bW_0(x_0)$. 
Thus,
\begin{equation*}
\calH^m \left(E \cap g_{\bv_1,\ldots,\bv_{n-m},u}^{-1}\{y\} \right) \leq \left( \frac{2}{\sqrt{3}}\right)^m \balpha(m) (\rmdiam E)^m .
\end{equation*}
The proposition is proved.
\end{proof}

\begin{Corollary}
\label{cor.ub}
There are $\bdelta_{\theTheorem}(n,\Lambda) > 0$ and $\bc_{\theTheorem}(n) \geq 1$ with the following property. If $\rmdiam E \leq \bdelta_{\theTheorem}(n,\Lambda)$, then
\begin{equation*}
\calZ_E \bW(u) \leq \bc_{\theTheorem}(n) (\rmdiam E)^m,
\end{equation*}
for $\calL^n$-almost every $u \in E$.
\end{Corollary}

\begin{proof}
Let $\bdelta_{\ref{cor.ub}}(n,\Lambda) = \min\{ \bdelta_{\ref{upper.bound}}(n,\Lambda),\bdelta_{\ref{Z.1}}(n,\Lambda,1/2) \}$. 
\end{proof}

\section{Lower bound for $\calY_E \bW$ and $\calZ_E \bW$}

\begin{Empty}[Setting for this section]
\label{51}
We enforce again the exact same assumptions as in \ref{31}; and as in \ref{fibration.2}, we let $C_r = \R^{n-m} \cap \{ y : |y| \leq r \}$.
\end{Empty}

\begin{Empty}[Polyballs]
\label{pb}
Given $x_0 \in \Rn$ and $r > 0$, we define
\begin{equation*}
\bC_\bW(x_0,r) = \Rn \cap \left\{ x : \left| P_{\bW_0(x_0)}(x-x_0)\right| \leq r \text{ and } \left| P_{\bW_0(x_0)^\perp}(x-x_0)\right| \leq r \right\} .
\end{equation*}
We notice that, if $x \in \bC_\bW(x_0,r)$, then $|x-x_0| \leq r \sqrt{2}$; in particular, $\rmdiam \bC_\bW(x_0,r) \leq 2 \sqrt{2}r$. 
We also notice that $\calL^n\left( \bC_\bW(x_0,r)\right) = \balpha(m)\balpha(n-m)r^n$. 
\end{Empty}

\begin{Empty}
\label{53}
{\it 
Given $0 < \veps < 1/3$, there exists $\bdelta_{\theTheorem}(n,\Lambda,\veps) > 0$ with the following property. If
\begin{enumerate}
\item $0 < r < \bdelta_{\theTheorem}(n,\Lambda,\veps)$;
\item $u \in \bC_\bW(x_0,r) \subset U$;
\item $|g_{\bv_1,\ldots,\bv_{n-m},u}(x_0)| \leq (1-3\veps)r$;
\item $C \subset C_{\veps r}$ is closed; 
\end{enumerate}
then
\begin{equation*}
\calL^n \left( \bC_\bW(x_0,r) \cap g_{\bv_1,\ldots,\bv_{n-m},u}^{-1}(C)\right) \geq \frac{1}{1+\veps} \balpha(m) r^m \calL^{n-m}(C) .
\end{equation*}
}
\end{Empty}

\begin{Remark}
With hopes that the following will help the reader form a geometrical imagery: Under the circumstances \ref{53}, $\bC_\bW(x_0,r) \cap g_{\bv_1,\ldots,\bv_{n-m},u}^{-1}(C)$ may be seen as a ``nonlinear stripe'', ``horizontal'' with respect to $\bW_0(x_0)$, ``at height'' $g_{\bv_1,\ldots,\bv_{n-m},u}(x_0)$ with respect to $x_0$, and of ``width'' $C$.
\end{Remark}

\begin{proof}[Proof of \ref{53}]
Given $z \in \bW_0(x_0) \cap \bB(0,r)$, we define 
\begin{equation*}
V_z = \Rn \cap \left\{ x_0 + z + \sum_{i=1}^{n-m} y_i \bv_i(x_0) : y \in C_r \right\} \subset \bC_\bW(x_0,r)
\end{equation*}
and we consider the isometric parametrization $\gamma_z : C_r \to V_z$ defined by the formula
\begin{equation*}
\gamma_z(y) = x_0 + z + \sum_{i=1}^{n-m} y_i \bv_i(x_0) .
\end{equation*}
We also abbreviate $f_{z,u} = g_{\bv_1,\ldots,\bv_{n-m},u} \circ \gamma_z$.
\par 
\textsc{Claim \#1.} $\rmLip f_{z,u} \leq (1 + \veps)^\frac{1}{n-m}$.
\par 
Since $\gamma_z$ is an isometry, it suffices to obtain an upper bound for $\rmLip g_{\bv_1,\ldots,\bv_{n-m},u}|_{\bC_\bW(x_0,r)}$. 
Let $x,x' \in \bC_\bW(x_0,r)$ and note that
\begin{equation*}
\begin{split}
\big| g_{\bv_1,\ldots,\bv_{n-m},u}(x)& - g_{\bv_1,\ldots,\bv_{n-m},u}(x') \big|  = \sqrt{\sum_{i=1}^{n-m} \left| \la \bv_i(x),x-u\ra - \la \bv_i(x'),x'-u\ra \right|^2} \\
& \leq \sqrt{\sum_{i=1}^{n-m} \left( \left| \la \bv_i(x)-\bv_i(x'),x-u\ra \right| + \left| \la \bv_i(x'),x-x' \ra \right|\right)^2} \\
&\leq \sqrt{\sum_{i=1}^{n-m}\left| \la \bv_i(x)-\bv_i(x'),x-u\ra \right|^2} + \sqrt{\sum_{i=1}^{n-m}\left| \la \bv_i(x'),x-x' \ra \right|^2}\\
& \leq \sqrt{n-m} \Lambda |x-x'||x-u| + \left| P_{\bW_0(x')^\perp}(x-x') \right| \\
& \leq \left( 1 + \sqrt{n-m} \Lambda 2 \sqrt{2}r \right) |x-x'| .
\end{split}
\end{equation*}
Recalling hypothesis (1), it is now apparent that $\bdelta_{\ref{53}}$ can be chosen small enough, according to $n$, $\Lambda$, and $\veps$, so that \textsc{Claim \#1} holds.
\par 
\textsc{Claim \#2.} {\it For $\calL^{n-m}$-almost every $y \in C_r$, one has $\|Df_{z,u}(y) - \rmid_{\R^{n-m}}\| \leq \veps$}.
\par 
Let $y \in C_r$ be such that $f_{z,u}$ is differentiable at $y$.
We shall estimate the coefficients of the matrix representing $Df_{z,u}(y)$ with respect to the canonical basis. 
Fix $i,j=1,\ldots,n-m$ and recall \eqref{eq.2}:
\begin{equation*}
\begin{split}
\frac{\partial}{\partial y_i} \la f_{u,z}(y) , e_j \ra & = \frac{\partial}{\partial y_i} \la g_{\bv_1,\ldots,\bv_{n-m},u}(\gamma_z(y)) , e_j \ra \\
& = \left(\frac{\partial }{\partial y_i}\right)g_{\bv_j,u}(\gamma_z(y)) \\
& = \left\la \nabla g_{\bv_j,u}(\gamma_z(y)) , \frac{\partial \gamma_z(y)}{\partial y_i} \right\ra \\
& = \left\la D\bv_j(\gamma_z(y))(\bv_i(x_0)) , \gamma_z(y)-u \right\ra + \left\la \bv_j(\gamma_z(y)) , \bv_i(x_0) \right\ra \\
& = \rmI + \rmII .
\end{split}
\end{equation*}
Next, note that
\begin{multline*}
\left| \rmII - \delta_{ij} \right| = \left| \rmII - \la \bv_j(x_0) , \bv_i(x_0) \ra \right| = \left| \la \bv_j(\gamma_z(y)) - \bv_j(x_0) , \bv_i(x_0) \ra \right| \\
\leq \Lambda \left| \gamma_z(y) - x_0 \right| \leq \Lambda 2 \sqrt{2} r \leq \frac{\veps}{2(n-m)},
\end{multline*}
where the last inequality follows from hypothesis (1), upon choosing $\bdelta_{\ref{53}}$ small enough, according to $n$, $\Lambda$, and $\veps$. 
Moreover,
\begin{equation*}
\left| \rmI \right| \leq \Lambda | \gamma_z(y)-u| \leq \Lambda 2 \sqrt{2} r \leq \frac{\veps}{2(n-m)}.
\end{equation*}
Therefore, if $(a_{ij})_{i,j=1,\ldots,n-m}$ is the matrix representing $Df_{z,u}(y)$ with respect to the canonical basis, we have shown that $|a_{ij}-\delta_{ij}| \leq \frac{\veps}{n-m}$, for all $i,j=1,\ldots,n-m$. 
This completes the proof of \textsc{Claim \#2}.
\par 
\textsc{Claim \#3.} $C_{\veps r } \subset f_{z,u}(C_r)$.
\par 
We shall show that $|y - f_{z,u}(y)| \leq  (1-\veps)r$, for every $y \in \rmBdry C_r$, and the conclusion will become a consequence of the intermediate value theorem, in case $m=n-1$, or a standard application of homology theory, see, e.g., \cite[4.6.1]{DEP.05c}, in case $m < n-1$. 
If $m < n-1$, it is clearly enough to establish this inequality only for $\calH^{n-m-1}$-almost every $y \in \rmBdry C_r$. 
In that case, owing to the coarea theorem \cite[3.2.22]{GMT}, we may choose such a $y$ so that that $f_{z,u}$ is differentiable $\calH^1$-almost everywhere on the line segment $\R^{n-m} \cap \{ sy : 0 \leq s \leq 1 \}$. 
Whether $m < n-1$, or $m=n-1$, it then follows from \textsc{Claim \#2} that
\begin{multline*}
\left| f_{z,u}(y) - f_{z,u}(0) - y  \right| = \left| \int_0^1 Df_{z,u}(sy)(y) d\calL^1(s) - y \right| \\ \leq \int_0^1 \left| Df_{z,u}(sy)(y)-y\right| d\calL^1(s) \leq \veps |y| = \veps r .
\end{multline*}
Accordingly,
\begin{equation*}
\left|f_{z,u}(y) - y \right| \leq \left| f_{z,u}(y) - f_{z,u}(0) - y  \right| + \left| f_{z,u}(0)\right| \leq \veps r + \left| f_{z,u}(0)\right|
\end{equation*}
and the claim will be established upon showing that $\left| f_{z,u}(0)\right| \leq (1- 2 \veps)r$.
Note that $f_{z,u}(0) = g_{\bv_1,\ldots,\bv_{n-m},u}(x_0+z)$; we shall use hypothesis (3) to bound its norm from above. 
Given $j=1,\ldots,n-m$, recall that $\la \bv_j(x_0) , z \ra = 0$, thus,
\begin{multline*}
\left| g_{\bv_j,u}(x_0+z) - g_{\bv_j,u}(x_0) \right| = \left| \la \bv_j(x_0+z) , x_0+z-u \ra - \la \bv_j(x_0) , x_0 - u \ra \right| \\
 = \left| \la \bv_j(x_0+z) , x_0+z-u \ra - \la \bv_j(x_0) , x_0 + z - u \ra \right| \leq \Lambda |z| |x_0 + z - u| \\
 \leq \Lambda r 2\sqrt{2} r \leq \frac{\veps r}{\sqrt{n-m}},
\end{multline*}
where the last inequality holds, according to hypothesis (1), provided $\bdelta_{\ref{53}}$ is chosen sufficiently small. 
In turn,
\begin{multline*}
\left| f_{z,u}(0) \right| \leq \left| g_{\bv_1,\ldots,\bv_{n-m},u}(x_0+z) - g_{\bv_1,\ldots,\bv_{n-m},u}(x_0) \right| + \left| g_{\bv_1,\ldots,\bv_{n-m},u}(x_0)\right| \\
\leq \veps r + (1-3\veps)r = (1-2\veps ) r,
\end{multline*}
by virtue of hypothesis (3).
\par 
\textsc{Claim \#4.} {\it For every $z \in \bW_0(x_0) \cap \bB(0,r)$ and every closed $C \subset C_{\veps r}$, one has $\calH^{n-m}(C) \leq (1+\veps) \calH^{n-m} \left( g_{\bv_1,\ldots,\bv_{n-m},u}^{-1}(C) \cap V_z \right)$.}
\par 
First, notice that
\begin{equation*}
g_{\bv_1,\ldots,\bv_{n-m},u}^{-1}(C) \cap V_z = \gamma_z \bigg( \gamma_z^{-1} \left(g_{\bv_1,\ldots,\bv_{n-m},u}^{-1}(C) \cap V_z \right) \bigg) = \gamma_z \left( f_{z,u}^{-1}(C)\right)
\end{equation*}
and, therefore,
\begin{equation*}
\calH^{n-m} \left(g_{\bv_1,\ldots,\bv_{n-m},u}^{-1}(C) \cap V_z \right) = \calH^{n-m}  \left( f_{z,u}^{-1}(C)\right),
\end{equation*}
since $\gamma_z$ is an isometry. 
Now, since $C \subset C_{\veps r} \subset f_{z,u}(C_r)$, according to \textsc{Claim \#3}, we have
\begin{equation*}
C = f_{z,u} \left( f_{z,u}^{-1}(C) \right) .
\end{equation*}
It therefore follows from \textsc{Claim \#1} that
\begin{equation*}
\begin{split}
\calH^{n-m}(C) & \leq \left( \rmLip f_{z,u} \right)^{n-m} \calH^{n-m}  \left( f_{z,u}^{-1}(C)\right) \\
& \leq (1+\veps) \calH^{n-m} \left(g_{\bv_1,\ldots,\bv_{n-m},u}^{-1}(C) \cap V_z \right).
\end{split}
\end{equation*}
\par 
We are now ready to finish the proof, by an application of Fubini's theorem:
\begin{equation*}
\begin{split}
\calL^n \bigg( \bC_\bW(x_0,r) \cap & g_{\bv_1,\ldots,\bv_{n-m},u}^{-1}(C)\bigg)\\& = \int_{\bW_0(x_0) \cap \bB(0,r)} d\calL^m(z) \calH^{n-m}\left(g_{\bv_1,\ldots,\bv_{n-m},u}^{-1}(C) \cap V_z \right) \\
& \geq \frac{1}{1+\veps} \balpha(m)r^m \calH^{n-m}(C) .
\end{split}
\end{equation*}
\end{proof}

\begin{Empty}[Lower bound for $\calY_E \bW$]
\label{54}
{\it 
Given $0 < \veps < 1/3$, there exists $\bdelta_{\theTheorem}(n,\Lambda,\veps)>0$ with the following property. If
\begin{enumerate}
\item $0 < r < \bdelta_{\theTheorem}(n,\Lambda,\veps)$;
\item $\bC_\bW(x_0,r) \subset U$;
\item $A \subset U$ is closed;
\item $\calL^n \left( A \cap \bC_\bW(x_0,r) \right) \geq (1 - \veps) \calL^n \left(\bC_\bW(x_0,r) \right)$;
\end{enumerate}
then
\begin{equation*}
\int_{\bC_\bW(x_0,r)} \calY_{A \cap \bC_\bW(x_0,r)} \bW (u) d\calL^n(u) \geq (1 - \bc_{\theTheorem}(n) \veps) \balpha(m)r^m\calL^n \left( \bC_\bW(x_0,r) \right) ,
\end{equation*}
where $\bc_{\theTheorem}(n)$ is a positive integer depending only upon $n$.
}
\end{Empty}

\begin{proof}
Similarly to the proof of \ref{53}, we will first establish a lower bound for $\calY_{A \cap \bC_\bW(x_0,r)} \bW$ on ``vertical slices'' $V_z$ of the given polyball, followed, next, by an application of Fubini. 
Given $z \in \bW_0(x_0) \cap \bB(0,r)$, we let $V_z$ and $\gamma_z$ be as in \ref{53}, and we consider the set
\begin{equation*}
\check{V}_z = \Rn \cap \left\{ x_0 + z + \sum_{i=1}^{n-m} y_i \bv_i(x_0) : y \in C_{(1-3\veps)r}\right\}
\end{equation*}
(notice $\check{V}_z$ is slightly smaller than $V_z$ used in the proof of \ref{53}) and the isometric parametrization $\check{\gamma}_z : C_{(1-3\veps)r} \to \check{V}_z$ defined by
\begin{equation*}
\check{\gamma}_z(y) = x_0 + z + \sum_{i=1}^{n-m} y_i \bv_i(x_0) .
\end{equation*}
For part of the proof, we find it convenient to abbreviate $E = A \cap \bC_\bW(x_0,r)$.
We also put $\check {\calY}_{E} \bW = \left( \calY_{E} \bW \right) \circ \check{\gamma}_z$.
\par 
By definition of $\calY_E \bW$, for each $\check{\gamma}_z(y) \in \check{V}_z$, there exists a collection $\calC_y$ of closed balls in $\R^{n-m}$ with the following properties: For every $C \in \calC_y$, $C$ is a ball centered at 0, $C \subset C_{\veps r}$,
\begin{equation*}
\calY_E \bW \left(\check{\gamma}_z(y) \right) + \veps \geq \dashint_C \calH^m \left( E \cap g_{\bv_1,\ldots,\bv_{n-m},\check{\gamma}_z(y)}^{-1}\{h\} \right) d\calL^{n-m}(h) ,
\end{equation*}
and $\inf \{ \rmdiam C : C \in \calC_y \} = 0$.
Furthermore, $\check{\calY}_E \bW$ being $\calL^{n-m}$-summable, according to \ref{upper.bound}, there exists $N \subset C_{(1-3\veps)r}$ such that $\calL^{n-m}(N)=0$ and every $y \not \in N$ is a Lebesgue point of $\check{\calY}_E \bW$. 
For such a $y$, we may reduce $\calC_y$, if necessary, keeping all the previously stated properties valid, while enforcing also that
\begin{equation*}
\dashint_{y+C} \check{\calY}_E \bW d\calL^{n-m} + \veps \geq \left(\check{\calY}_E \bW \right)(y),
\end{equation*}
whenever $C \in \calC_y$.
We infer that, for each $y \in C_{(1-3\veps)r} \setminus N$ and each $C \in \calC_y$,
\begin{equation}
\label{eq.30}
\int_{y + C} \check{\calY}_E \bW d\calL^{n-m} + 2\veps \calL^{n-m}(y+C) \geq \int_C \calH^m \left( E \cap g_{\bv_1,\ldots,\bv_{n-m},\check{\gamma}_z(y)}^{-1}\{h\} \right) d\calL^{n-m}(h).
\end{equation}
It ensues from the Vitali covering theorem that there is a sequence $(y_k)_k$ in $C_{(1-3\veps)r} \setminus N$, and $C_k \in \calC_{y_k}$, such that the balls $y_k + C_k$, $k=1,2,\ldots,$ are pairwise disjoint and $\calL^{n-m} \left( C_{(1-3\veps)r} \setminus \cup_{k=1}^\infty (y_k+C_k) \right) = 0$. 
It therefore follows, from \eqref{eq.30} and the fact that $\gamma_z$ is an isometry, that
\begin{equation}
\label{eq.31}
\int_{V_z} \calY_E \bW d\calH^{n-m} +2\veps \calH^{n-m}(V_z) \geq \sum_{k=1}^\infty \int_{C_k} \calH^m \left( E \cap g_{\bv_1,\ldots,\bv_{n-m},u_k}^{-1}\{y\} \right) d\calL^{n-m}(y),
\end{equation}
where we have abbreviated $u_k = \check{\gamma}_z(y_k)$. 
We also abbreviate $S_k = g_{\bv_1,\ldots,\bv_{n-m},u_k}^{-1}(C_k)$ and we infer from the coarea formula that, for each $k=1,2,\ldots$,
\begin{multline}
\label{eq.32}
\int_{C_k} \calH^m \left( E \cap g_{\bv_1,\ldots,\bv_{n-m},u_k}^{-1}\{y\} \right) d\calL^{n-m}(y) = \int_{E \cap S_k} Jg_{\bv_1,\ldots,\bv_{n-m},u_k}d\calL^n \\
\geq (1-\veps) \calL^n \left( E \cap S_k\right),
\end{multline}
where the last inequality follows from \ref{jac.g}, applied with $U = \rmInt \bC_\bW(x_0,r)$, provided that $\bdelta_{\ref{54}}(n,\Lambda,\veps)$ is chosen smaller than $(2\sqrt{2})^{-1}\bdelta_{\ref{jac.g}}(n,\Lambda,\veps)$. 
Letting $S = \cup_{k=1}^\infty S_k$ and recalling that $E = A \cap \bC_\bW(x_0,r)$, we infer, from \eqref{eq.31} and \eqref{eq.32}, that
\begin{multline}
\label{eq.33}
\int_{V_z} \calY_E \bW d\calH^{n-m} +2\veps \calH^{n-m}(V_z) \geq (1-\veps) \calL^n(E \cap S) \\
 \geq (1-\veps) \big( \calL^n ( \bC_\bW(x_0,r) \cap S) - \calL^n( \bC_\bW(x_0,r) \setminus A ) \big).
\end{multline}
Applying \ref{53} to each $S_k$ does not immediately yield a lower bound for $\calL^n ( \bC_\bW(x_0,r) \cap S)$, because the $S_k$ are not necessarily pairwise disjoint. 
This is why we now introduce slightly smaller versions of these:
\begin{equation*}
\check{C}_k = (1-\veps) C_k \quad \text{ and } \quad \check{S}_k = g_{\bv_1,\ldots,\bv_{n-m},u_k}^{-1}\left(\check{C}_k\right) .
\end{equation*}
\par 
\textsc{Claim.} {\it The sets $\check{S}_k \cap \bC_\bW(x_0,r)$, $k=1,2,\ldots$, are pairwise disjoint.}
\par
Assume, if possible, that there are $j \neq k$ and $x \in \check{S}_j \cap \check{S}_k \cap \bC_\bW(x_0,r)$. 
Letting $\rho_j$ and $\rho_k$ denote, respectively, the radius of $C_j$, and of $C_k$, we notice that $\rho_j + \rho_k < |y_j-y_k|$, because $(y_j+C_j) \cap (y_k+C_k) = \emptyset$. 
Since $\check{\gamma}_z$ is an isometry, we have $|u_j-u_k| = \left| \check{\gamma}_z(y_j) - \check{\gamma}_z(y_k)\right| = |y_j-y_k|$ and, therefore, also
\begin{multline}
\label{eq.34}
\left| g_{\bv_1,\ldots,\bv_{n-m},u_j}(x) - g_{\bv_1,\ldots,\bv_{n-m},u_k}(x)\right| \leq \left| g_{\bv_1,\ldots,\bv_{n-m},u_j}(x) \right| + \left| g_{\bv_1,\ldots,\bv_{n-m},u_k}(x) \right|\\
\leq (1-\veps)\rho_j + (1-\veps)\rho_k < (1-\veps) \left| u_j - u_k \right| .
\end{multline}
We now introduce the following vectors of $\R^{n-m}$:
\begin{equation*}
h_j = \sum_{i=1}^{n-m} \la \bv_i(x_0),u_j \ra e_i \quad \text{ and } \quad h_k = \sum_{i=1}^{n-m} \la \bv_i(x_0),u_k \ra e_i
\end{equation*}
and we notice that
\begin{equation*}
\left| h_j - h_k \right| = \left| P_{\bW_0(x_0)^\perp}(u_j-u_k)\right| = |u_j-u_k|,
\end{equation*}
where the second equality holds because $u_j-u_k \in \bW_0(x_0)^\perp$, as clearly follows from the definition of $\check{\gamma}_z$. 
Furthermore,
\begin{multline*}
\bigg| \big( g_{\bv_1,\ldots,\bv_{n-m},u_j}(x) - g_{\bv_1,\ldots,\bv_{n-m},u_k}(x)\big) - \big( h_k - h_j \big)  \bigg| 
= \sqrt{\sum_{i=1}^{n-m} \left| \la \bv_i(x) - \bv_i(x_0) , u_k-u_j \ra\right|^2}\\
\leq \sqrt{n-m} \Lambda \sqrt{2} r \left| u_j - u_k \right| \leq \veps \left| u_j - u_k \right| ,
\end{multline*}
since we may choose $\bdelta_{\ref{54}}(n,\Lambda,\veps)$ to be so small that the last inequality holds, in view of hypothesis (1). 
Whence,
\begin{equation*}
\left| g_{\bv_1,\ldots,\bv_{n-m},u_j}(x) - g_{\bv_1,\ldots,\bv_{n-m},u_k}(x)\right| \geq \left| h_j - h_k \right| -  \veps \left| u_j - u_k \right| = (1-\veps)  \left| u_j - u_k \right|,
\end{equation*}
contradicting \eqref{eq.34}. 
The \textsc{Claim} is established.
\par 
Thus,
\begin{equation}
\label{eq.35}
\begin{split}
\calL^n \left( \bC_\bW(x_0,r) \cap S\right) & = \calL^n \left( \bC_\bW(x_0,r) \cap \cup_{k=1}^\infty S_k\right) \\
& \geq \calL^n \left( \bC_\bW(x_0,r) \cap \cup_{k=1}^\infty \check{S}_k\right) \\
& = \sum_{k=1}^\infty \calL^n \left( \bC_\bW(x_0,r) \cap  \check{S}_k\right) \\
& = \sum_{k=1}^\infty \calL^n \left( \bC_\bW(x_0,r) \cap g_{\bv_1,\ldots,\bv_{n-m},u_k}^{-1}\left(\check{C}_k\right) \right) \\
& \geq \frac{1}{1+\veps}\balpha(m)r^m \sum_{k=1}^\infty \calL^{n-m} \left( \check{C}_k \right),
\end{split}
\end{equation}
where the last inequality follows from \ref{53}. 
We notice that, indeed, \ref{53} applies, since $\check{C}_k \subset C_k \subset C_{\veps r}$ and $\left|g_{\bv_1,\ldots,\bv_{n-m},u_k}(x_0)\right| = \left| P_{\bW_0(x_0)^\perp}(u_k-x_0) \right| = \left| y_k\right| \leq (1-3\veps)r$.
\par 
Now,
\begin{multline}
\label{eq.36}
\sum_{k=1}^\infty \calL^{n-m} \left( \check{C}_k \right)  = (1-\veps)^{n-m} \sum_{k=1}^\infty \calL^{n-m} \left( C_k \right) = (1-\veps)^{n-m} \sum_{k=1}^\infty \calL^{n-m} \left( y_k + C_k \right) \\
\geq (1-\veps)^{n-m} \calL^{n-m} \left( C_{(1-3\veps)r} \right) \geq (1-3\veps)^{2(n-m)} \balpha(n-m)r^{n-m} .
\end{multline}
We infer, from \eqref{eq.35} and \eqref{eq.36}, that
\begin{equation*}
\calL^n \left( \bC_\bW(x_0,r) \cap S\right) \geq \frac{(1-3\veps)^{2(n-m)}}{1+\veps}\calL^n \left( \bC_\bW(x_0,r)\right) .
\end{equation*}
It therefore ensues, from \eqref{eq.33} and hypothesis (4), that
\begin{equation*}
\int_{V_z} \calY_E \bW d\calH^{n-m} +2\veps \calH^{n-m}(V_z) \geq (1-\veps)\left(\frac{(1-3\veps)^{2(n-m)}}{1+\veps} - \veps \right)\calL^n \left( \bC_\bW(x_0,r)\right) .
\end{equation*}
Integrating over $z$, we infer from Fubini's theorem that
\begin{multline*}
\int_{\bC_\bW(x_0,r)} \calY_{A \cap \bC_\bW(x_0,r)} \bW d\calL^n = \int_{\bW_0(x_0) \cap \bB(0,r)} d\calL^n(z) \int_{V_z} \calY_E \bW d\calH^{n-m} \\
\geq \left[ (1-\veps)\left(\frac{(1-3\veps)^{2(n-m)}}{1+\veps}-\veps\right)-2\veps\right]\balpha(m)r^m\calL^n \left( \bC_\bW(x_0,r)\right) .
\end{multline*}
\end{proof}

\begin{Proposition}
\label{lower.bound}
Given $0 < \veps < 1/3$, there exist $\bdelta_{\theTheorem}(n,\Lambda,\veps)>0$ and $\bc_{\theTheorem}(n) \geq 1$ with the following property. If
\begin{enumerate}
\item $0 < r < \bdelta_{\theTheorem}(n,\Lambda,\veps)$;
\item $\bC_\bW(x_0,r) \subset U$;
\item $A \subset U$ is closed;
\item $\calL^n \left( A \cap \bC_\bW(x_0,r) \right) \geq (1 - \veps) \calL^n \left(\bC_\bW(x_0,r) \right)$;
\end{enumerate}
then
\begin{equation*}
\int_{A \cap \bC_\bW(x_0,r)} \calY_{A \cap \bC_\bW(x_0,r)} \bW (u) d\calL^n(u) \geq (1 - \bc_{\theTheorem}(n) \veps) \balpha(m)r^m\calL^n \left( \bC_\bW(x_0,r) \right) .
\end{equation*}
\end{Proposition}

\begin{Remark}
The difference with \ref{54} is the domain of integration (being smaller) in the integral, on the left hand side in the conclusion.
\end{Remark}

\begin{proof}[Proof of \ref{lower.bound}]
The reader will happily check that 
\begin{equation*}
\bdelta_{\ref{lower.bound}}(n,\Lambda,\veps) = \min \left\{ \bdelta_{\ref{54}}(n,\Lambda,\veps), \left(2\sqrt{2}\right)^{-1} \bdelta_{\ref{upper.bound}}(n,\Lambda)\right\}
\end{equation*}
suits their needs.
\end{proof}

\begin{Proposition}
\label{Z.positive}
There exists $\bdelta_{\theTheorem}(n,\Lambda) > 0$ with the following property. If $\rmdiam E \leq \bdelta_{\theTheorem}(n,\Lambda)$, then
\begin{equation*}
\calZ_E \bW(u) > 0,
\end{equation*}
for $\calL^n$-almost every $u \in E$.
\end{Proposition}

\begin{proof}
We let 
\begin{equation*}
\bdelta_{\ref{Z.positive}}(n,\Lambda) = \min \left\{ \bdelta_{\ref{lower.bound}}\left(n,\Lambda,\frac{1}{4\bc_{\ref{lower.bound}}(n)}\right) , \bdelta_{\ref{lb.2}}(n,\Lambda,1/2) \right\} .
\end{equation*}
According to \ref{lb.2}, it suffices to show that $\calY_E \bW (u) > 0$, for $\calL^n$-almost every $u \in E$. 
Define $Z = E \cap \{ u : \calY_E \bW(u) = 0 \}$ and assume, if possible, that $\calL^n(Z) > 0$.
Since $Z$ is $\calL^n$-measurable (recall \ref{def.Y}), there exists a compact set $A \subset Z$ such that $\calL^n(A) > 0$. 
Observe that the sets $\bC_\bW(x,r)$, for $x \in U$ and $r > 0$, form a density basis for $\calL^n$-measurable subsets of $U$ -- because their eccentricity is bounded away from zero -- thus, there exists $x_0 \in A$ and $r_0 > 0$ such that
\begin{equation*}
\calL^n \left( A \cap \bC_\bW(x_0,r) \right) \geq \left(1 - \frac{1}{4\bc_{\ref{lower.bound}}(n)} \right) \calL^n \left( \bC_\bW(x_0,r)\right),
\end{equation*}
whenever $0 < r < r_0$. 
There is no restriction to assume that $r_0$ is small enough for $\bC_\bW(x_0,r_0) \subset U$. 
Thus, if we let $r= \min \{ r_0 , \bdelta_{\ref{lower.bound}}(n,\Lambda,1/(4 \bc_{\ref{lower.bound}(n)}))\}$, it follows from \ref{lower.bound} that
\begin{equation}
\label{eq.40}
\int_{A \cap \bC_\bW(x_0,r)} \calY_{A \cap \bC_\bW(x_0,r)} \bW (u) d\calL^n(u) \geq \left(1 -  \frac{1}{4}\right) \balpha(m)r^m\calL^n \left( \bC_\bW(x_0,r) \right) > 0 .
\end{equation}
On the other hand, recalling \ref{def.Y} and the fact that $A \cap \bC_\bW(x_0,r) \subset E$, we infer that $\calY_{A \cap \bC_\bW(x_0,r)} \bW(u) \leq \calY_E\bW(u)$, for all $u \in \Rn$. 
In particular, $\calY_{A \cap \bC_\bW(x_0,r)} \bW(u)=0$, for all $u \in A \cap \bC_\bW(x_0,r) \subset Z$, contradicting \eqref{eq.40}. 
\end{proof}

\section{Proof of the theorems}

\begin{Theorem}
Assume that $S \subset \Rn$, $\bW_0 : S \to \bG(n,m)$ is Lipschitzian, and $A \subset S$ is Borel measurable. The following are equivalent.
\begin{enumerate}
\item $\calL^n(A)=0$.
\item For $\calL^n$ almost every $x \in A$, $\calH^m(A \cap \bW(x))=0$.
\item For $\calL^n$ almost every $x \in S$, $\calH^m(A \cap \bW(x))=0$.
\end{enumerate}
\end{Theorem}

Recall our convention that $\bW(x) = x + \bW_0(x)$.

\begin{proof}
Since $\bG(n,m)$ is complete, we can extend $\bW_0$ to the closure of $S$. 
Furthermore, if the theorem holds for $\rmClos S$, then it also holds for $S$. 
Thus, there is no restriction to assume that $S$ is closed.
\par 
$(1) \Rightarrow (3)$. It follows from \ref{orth.frame} that each $x \in S$ admits an open neighborhood $U_x$ in $\Rn$ such that $\bW(x)$ can be associated with a Lipschitzian orthonormal frame satisfying all the conditions of \ref{31}, for some $\Lambda_x > 0$. 
Since $S$ is Lindel\"of, there are countably many $x_1,x_2,\ldots$ such that $S \subset \cup_j U_{x_j}$. 
Letting $E_j = S \cap U_{x_j}$, we infer from \ref{AC.1} that $\phi_{E_j,\bW}$ is absolutely continuous with respect to $\calL^n$. 
Thus, if $\calL^n(A)=0$, then $\calH^m\left( A \cap \bW(x)\right)=0$, for $\calL^n$-almost every $x \in E_j$, by definition of $\phi_{E_j,\bW}$. 
Since $j$ is arbitrary, the proof is complete.
\par
$(3) \Rightarrow (2)$ is trivial.
\par 
$(2) \Rightarrow (1)$. Let $A$ satisfy condition (2). 
It is enough to show that $\calL^n(A \cap \bB(0,r)) = 0$, for each $r > 0$. 
Fix $r > 0$ and define $S_r = S \cap \bB(0,r)$. 
Consider the $U_{x_j}$ defined in the second paragraph of the present proof; since $S_r$ is compact, finitely many of those, say $U_{x_1},\ldots,U_{x_N}$, cover $S_r$. 
Let $\Lambda = \max_{j=1,\ldots,N} \Lambda_{x_j}$. 
Partition each $U_{x_j}$, $j=1,\ldots,N$, into Borel measurable sets $E_{j,k}$, $k=1,\ldots,
K_j$, such that $\rmdiam E_{j,k} \leq \bdelta_{\ref{Z.positive}}(n,\Lambda)$. 
It then follows from \ref{Z.positive} that
\begin{equation}
\label{eq.50}
\left( \calZ_{A \cap E_{j,k}} \bW \right)(u) > 0,
\end{equation}
for $\calL^n$-almost every $u \in A \cap E_{j,k}$. 
Now, fix $j$ and $k$. 
Observe that $\calH^m \left( A \cap E_{j,k} \cap \bW(x) \right) = 0$, for $\calL^n$-almost every $x \in A \cap E_{j,k}$. 
Thus, $\phi_{A \cap E_{j,k},\bW}(A \cap E_{j,k}) =0$. 
Moreover,
\begin{equation*}
0 = \phi_{A \cap E_{j,k},\bW}\left(A \cap E_{j,k}\right) = \int_{A \cap E_{j,k}} \left( \calZ_{A \cap E_{j,k}} \bW \right)(u) d\calL^n(u) .
\end{equation*}
It follows from \eqref{eq.50} that $\calL^n(A \cap E_{j,k}) =0$. 
Since $j$ and $k$ are arbitrary, $\calL^n(A)=0$.
\end{proof}

\begin{Remark}
\label{remark}
Alternatively, one can prove the principal implication $(2) \Rightarrow (1)$ in two other ways.
One way -- more involved -- consists in applying our main result \ref{main.density} below.
A second -- simpler -- way, along the following lines, avoids reference to the estimates we obtained for the functions $\calY_E \bW$ and $\calZ_E \bW$.
Consider $x \in A$, $\eta > 0$, and $V \in \bG(n,n-m)$ such that $d\left(V^\perp, \bW_0(x) \right) < \eta$, and put $V_x = x + V$.
Define
\begin{equation*}
\Phi : (U_x \cap V_x) \times \Rm \to \Rn : (\xi,t) \mapsto \xi + \sum_{i=1}^m t_i \bw_i(\xi) \,.
\end{equation*}
$\Phi$ is locally Lipschitzian and one checks that, in fact, $\Phi$ is a lipeomorphism between $B$ and $\Phi(B)$, where $B = \bB(x,\rho)$, for some $\rho > 0$ depending upon $\eta$ and $\Lambda_x$, because its differential is close to the identity.
Referring to Fubini's theorem, one then further checks that
\begin{equation*}
\calL^n(A') = 0 \text{ if and only if } \calH^m(A' \cap \bW(\xi)) = 0 \text{, for  $\calH^{n-m}$-almost every } \xi \in V_x.
\end{equation*}
Finally, using Fubini again, with respect to the decomposition $\Rn = \bW_0(x) \oplus \bW_0(x)^\perp$, one shows that $\calH^m(A \cap \bW(\zeta))=0$, for $\calH^{n-m}$-almost every $\zeta \in x' + \bW_0(x)^\perp$, where $x'$ is as close as we wish to $x$.
Applying the previous construction with $x'$ replacing $x$, using $V = \bW_0(x)^\perp$, we find that $\calL^n(A \cap \bB(x,r))=0$, for some $r > 0$ depending on $\Lambda_x$.
\par 
Notwithstanding, it seems that the (simpler) change of variable described here is not enough to yield the (stronger) theorem below. 
\end{Remark}

\begin{Empty}[Polyballs]
\label{pb.complement}
Recalling \ref{pb}, we notice that
\begin{equation*}
\bC_{\bW}(x_0,r) = \Rn \cap \left\{ \nu_{x_0}(x-x_0) \leq r \right\},
\end{equation*}
where $\nu_{x_0}$ is a norm on $\Rn$ defined by the formula
\begin{equation*}
\nu_{x_0}(x) = \max \left\{  \left| P_{\bW_0(x_0)}(x) \right| ,  \left| P_{\bW_0(x_0)^\perp}(x) \right| \right\},
\end{equation*} 
for $x \in \Rn$.
It is readily observed that $\rmLip \nu_{x_0} \leq 1$.
\begin{enumerate}
\item[(1)] {\it One has $|\nabla \nu_{x_0}(x) | = 1$, for $\calL^n$-almost every $x \in \Rn$.} 
\end{enumerate}
\par 
Abbreviate $P = P_{\bW_0(x_0)}$ and $Q = P_{\bW_0(x_0)^\perp}$ and define $S = \Rn \cap \{ x : |P(x)| = |Q(x)| \}$, so that $\calL^n(S)=0$.
Let $x \in \Rn \setminus S$ and notice $\nu_{x_0}$ is differentiable at $x$.
We henceforth assume that $|P(x)| < |Q(x)|$, whence, $\nu_{x_0}(x) = |Q(x)|$ -- the proof in the other case is similar.
Define $\veps = |Q(x)| - |P(x)| > 0$, $e = \frac{Q(x)}{|Q(x)|}$, and let $0 < t < \veps$. 
Note that $|Q(x + te)| = |Q(x)| \left( 1 + \frac{t}{|Q(x)|}\right) > |Q(x)|$ and $|P(x + te)| \leq |P(x)| + t < |Q(x)$.
Thus, $\nu_{x_0}(x+te) = |Q(x+te)|$ and, in turn, $\nu_{x_0}(x+te) - \nu_{x_0}(x) = |Q(x+te)| - |Q(x)| = t$.
It follows that $\la \nabla \nu_{x_0}(x) , e \ra = 1$.
Since $e$ is a unit vector and $\rmLip \nu_{x_0} \leq 1$, we conclude that $|\nabla \nu_{x_0}(x)| = 1$.
\begin{enumerate}
\item[(2)] {\it Let $U$, $\bW$ and $\Lambda > 0$ be as in \ref{31}. Assume $\bC_{\bW}(x_0,r) \subset U$, $x \in \bC_{\bW}(x_0,r)$, and $0 \leq t \leq 1$ is so that $\nu_{x_0}(x-x_0) = tr$. It follows that
\begin{equation*}
\bC_{\bW}(x_0,r) \cap \bW(x) \subset \bB \left( x , r(1+t) + 8m\Lambda  r^2\right) \cap \bW(x) .
\end{equation*}
 }
\end{enumerate}
\par 
First notice that $|z| \leq \sqrt{2} \nu_{x_0}(z)$, for every $z \in \Rn$, so that $|x-x_0| \leq \sqrt{2} \nu_{x_0}(x-x_0) = \sqrt{2} tr$ and, for every $x' \in \bC_\bW(x_0,r)$, $|x-x'| \leq \sqrt{2} \left( \nu_{x_0}(x-x_0) + \nu_{x_0}(x_0-x') \right) \leq \sqrt{2} r (1+t)$.
Next notice that, for each $h \in \Rn$,
\begin{multline*}
\left| \left( P_{\bW_0(x_0)} - P_{\bW_0(x)} \right) (h) \right| = \left| \sum_{i=1}^m \la \bw_i(x_0) , h \ra \bw_i(x_0) - \sum_{i=1}^m \la \bw_i(x) , h \ra \bw_i(x)\right| \\
\leq \sum_{i=1}^m \left| \la \bw_i(x_0),h \ra \right| \left| \bw_i(x_0) - \bw_i(x) \right| + \sum_{i=1}^m \left| \la \bw_i(x_0) - \bw_i(x) , h \ra \right| \left| \bw_i(x) \right| \\
\leq 2 m \Lambda \left| x-x_0 \right| |h| .
\end{multline*}
We now assume that $x' \in \bC_\bW(x,r) \cap \bW(x)$, in particular, $x-x' \in \bW_0(x)$, whence,
\begin{multline*}
\left| x-x' \right| = \left| P_{\bW_0(x)}(x-x')\right| \leq \left| P_{\bW_0(x_0)}(x-x')\right| + \left| \left( P_{\bW_0(x_0)} - P_{\bW_0(x)}\right)(x-x')\right| \\
\leq \nu_{x_0}(x-x_0) + \nu_{x_0}(x'-x_0) + 2 m \Lambda |x-x_0| |x-x'| \leq r(1+t) + 4m \Lambda t(1+t) r^2.
\end{multline*} 
\end{Empty}

\begin{Theorem}
\label{main.density}
Assume that $A \subset \Rn$ is Borel measurable and that $\bW_0 : A \to \bG(n,m)$ is Lipschitzian.
It follows that
\begin{equation*}
\limsup_{r \to 0^+} \frac{\calH^m\left(A \cap \bB(x,r) \cap \bW(x)\right)}{\balpha(m)r^m} \geq \frac{1}{2^n},
\end{equation*}
for $\calL^n$-almost every $x \in A$.
\end{Theorem}

Recall our convention that $\bW(x) = x + \bW_0(x)$.

\begin{proof}
Extend $\bW_0$ to $\Rn$ in a Borel measurable way, for instance, to be an arbitrary constant outside of $A$.
It follows from \ref{second.measurability} that, for each $r > 0$, the function $\Rn \to [0,\infty] : x \mapsto \frac{\calH^m ( A \cap \bB(x,r) \cap \bW(x))}{\balpha(m)r^m}$ is Borel measurable.
Thus, for each $j=1,2,\ldots$, the function $g_j : \Rn \to [0,\infty]$ defined by
\begin{equation*}
%\begin{split}
g_j(x)  = \sup_{0 < r < \frac{1}{j}} \frac{\calH^m \left( A \cap \bB(x,r) \cap \bW(x) \right)}{\balpha(m)r^m}
 = \sup_{\substack{0 < r < \frac{1}{j}\\r \,\text{rational} }} \frac{\calH^m \left( A \cap \bB(x,r) \cap \bW(x) \right)}{\balpha(m)r^m}
%\end{split}
\end{equation*}
is Borel measurable as well, and so is $g = \lim_j g_j = \inf_j g_j$.
\par 
Abbreviate $\boldeta(n,m) = 2^{-(n-m)}$.
Arguing {\it reductio ad absurdum}, we henceforth assume that $A$ and $\bW_0$ fail the conclusion of the theorem.
Thus, the set $Z_0 = A \cap \left\{ x : g(x) < \frac{\boldeta(n,m)}{2^m}\right\}$ is Borel measurable and non Lebesgue null.
Accordingly, there exists $\veps_1 > 0$ such that the set $Z_0' = A \cap \left\{ x : g(x) < (1-\veps_1) \frac{\boldeta(n,m)}{(2+\veps_1)^m}\right\}$ is also Borel measurable and of positive Lebesgue measure.
It therefore ensues from Egoroff's theorem \cite[2.3.7]{GMT} that there exists a closed set $Z \subset Z'_0 \subset A$ such that $\calL^n(Z) > 0$ and that there exists a positive integer $j_0$ such that 
\begin{equation}
\label{final.1}
\frac{\calH^m\left( Z \cap \bB(x,r) \cap \bW(x)\right)}{\balpha(m)r^m} \leq g_{j_0}(x) < (1-\veps_1) \frac{\boldeta(n,m)}{(2 + \veps_1)^m},
\end{equation}
for each $x \in Z$ and each $0 < r < \frac{1}{j_0}$.
Choose $0 < \veps_2 < 10$ such that 
\begin{equation}
\label{final.2}
\frac{1-\veps_1}{1-\veps_2} \leq 1 - \frac{\veps_1}{2}
\end{equation}
and choose $0 < \veps_3 < 1/3$ such that
\begin{equation}
\label{final.3}
1 - \frac{\veps_1}{2} < 1 - \bc_{\ref{lower.bound}}(n) \veps_3 .
\end{equation}
\par 
As in the proof of \ref{Z.positive}, we recall that the family $\bC_\bW(x,r)$, for $x \in \Rn$ and $r > 0$, is a density basis of $\calL^n$-measurable sets. 
Since $\calL^n(Z) > 0$, there exists $x_0 \in Z$ such that
\begin{equation*}
\lim_{r \to 0^+} \frac{\calL^n \left( Z \cap \bC_\bW(x_0,r)\right)}{\calL^n \left( \bC_\bW(x_0,r) \right)} = 1 .
\end{equation*}
In particular, there exists $R > 0$ such that 
\begin{equation}
\label{final.4}
(1- \veps_3) \calL^n \left( \bC_\bW(x_0,r)\right) \leq \calL^n \left( Z \cap \bC_\bW(x_0,r)\right),
\end{equation}
whenever $0 < r < R$.
\par 
We let $U$ be an open neighborhood of $x_0$ in $\Rn$ associated with $A$ and $\bW_0$ in \ref{orth.frame}, so that $\bW_0$ and $\bW_0^\perp$ are associated with orthonormal frames as in \ref{31}, for some $\Lambda > 0$.
Define
\begin{equation*}
r_0 = \min \bigg\{ 1, \frac{1}{j_0(2+8m\Lambda)} , \frac{\veps_1}{8m\Lambda} , \frac{\bdelta_{\ref{lb.1}}(n,\Lambda,\veps_2)}{2\sqrt{2}} , \frac{\rmdist(x_0,\Rn \setminus U)}{2\sqrt{2}} , \bdelta_{\ref{lower.bound}}(n,\Lambda,\veps_3) , R \bigg\} .
\end{equation*}
Let $0 < r < r_0$ and observe that
\begin{equation}
\label{final.5}
\begin{split}
\left( 1 - \bc_{\ref{lower.bound}}(n) \veps_3\right) & \calL^n \left( \bC_\bW(x_0,r)\right)  \leq \int_{Z \cap \bC_\bW(x_0,r)} \frac{\calY_{Z \cap \bC_\bW(x_0,r)} \bW(u)}{\balpha(m)r^m} d\calL^n(u) \intertext{(by \ref{lower.bound} applied with $\veps=\veps_3$ and $A=Z \cap \bC_\bW(x_0,r)$)}
& \leq \frac{1}{(1-\veps_2) \boldeta(n,m)} \int_{Z \cap \bC_\bW(x_0,r)} \frac{\calH^m \left( Z \cap \bC_\bW(x_0,r) \cap \bW(x) \right)}{\balpha(m)r^m} d\calL^n(x)
\end{split}
\end{equation}
(by \ref{lb.1}, applied with $\veps=\veps_2$ and $E=B=Z\cap\bC_\bW(x_0,r)$).
We also note that
\begin{equation}
\label{final.6}
\begin{split}
 &\int_{Z \cap \bC_\bW(x_0,r)}  \frac{\calH^m \left( Z \cap \bC_\bW(x_0,r) \cap \bW(x) \right)}{\balpha(m)r^m} d\calL^n(x) \\
 & = \int_{Z \cap \{ \nu_{x_0} \leq r \}}  \frac{\calH^m \left( Z \cap \bC_\bW(x_0,r) \cap \bW(x) \right)}{\balpha(m)r^m} \left| \nabla \nu_{x_0}(x)\right| d\calL^n(x) 
 \intertext{(by \ref{pb.complement}(1))}
 & = \int_0^r d\calL^1(\rho) \int_{Z \cap \{ \nu_{x_0} = \rho \}} \frac{\calH^m \left( Z \cap \bC_\bW(x_0,r) \cap \bW(x) \right)}{\balpha(m)r^m} d\calH^{n-1}(x)
 \intertext{(by \cite[3.4.3]{EVANS.GARIEPY})}
& = r \int_0^1 d\calL^1(t) \int_{Z \cap \{ \nu_{x_0} = tr \}} \frac{\calH^m \left( Z \cap \bC_\bW(x_0,r) \cap \bW(x) \right)}{\balpha(m)r^m} d\calH^{n-1}(x) \\
& \leq r \int_0^1 d\calL^1(t) \int_{Z \cap \{ \nu_{x_0} = tr \}} \frac{\calH^m \left( Z \cap \bB\left(x , r(1+t) + 8m\Lambda r^2\right) \cap \bW(x) \right)}{\balpha(m)r^m} d\calH^{n-1}(x)
\intertext{(by \ref{pb.complement}(2))}
& \leq r \int_0^1 d\calL^1(t) \int_{Z \cap \{ \nu_{x_0} = tr \}} (1-\veps_1) \frac{\boldeta(n,m)}{(2+\veps_1)^m}(1+t + 8m\Lambda r)^m d\calH^{n-1}
\intertext{(by \eqref{final.1})}
& \leq (1-\veps_1) \boldeta(n,m) \int_0^r d\calL^1(\rho) \int_{Z \cap \{ \nu_{x_0} = \rho \}}  d\calH^{n-1} \\
& = (1-\veps_1) \boldeta(n,m) \int_{Z \cap \{ \nu_{x_0} \leq r \}} \left| \nabla \nu_{x_0} (x)\right| d\calL^n(x)
 \intertext{(by \cite[3.4.3]{EVANS.GARIEPY})}
& = (1-\veps_1) \boldeta(n,m) \calL^n \left( Z \cap \bC_\bW(x_0,r)\right) \leq (1-\veps_1) \boldeta(n,m) \calL^n \left(\bC_\bW(x_0,r)\right) 
\end{split}
\end{equation}
(by \ref{pb.complement}(1)).
Plugging \eqref{final.6} into \eqref{final.5}, we obtain
\begin{equation*}
\begin{split}
\left( 1 - \bc_{\ref{lower.bound}}(n) \veps_3\right) \calL^n \left( \bC_\bW(x_0,r)\right)  & \leq \left( \frac{1-\veps_1}{1-\veps_2}\right) \calL^n \left(  \bC_\bW(x_0,r)\right) \\
& < \left( 1 - \bc_{\ref{lower.bound}}(n) \veps_3 \right) \calL^n \left( \bC_\bW(x_0,r)\right) 
\end{split}
\end{equation*}
(by \eqref{final.2} and \eqref{final.3}), a contradiction.
\end{proof}

\bibliographystyle{amsplain}
\bibliography{/home/thierry/Documents/LaTeX/Bibliography/thdp}

%\printindex

\end{document}